\documentclass[11pt, reqno]{amsart}

\usepackage[margin=0.9in]{geometry}
\usepackage[utf8]{inputenc}
\usepackage[T1]{fontenc}

\usepackage{subfiles}
\usepackage{comment}
\usepackage{float}
\usepackage{marginnote}
\usepackage{euscript}
\usepackage[dvipsnames]{xcolor}   		             		
\usepackage{graphicx}			
\usepackage{amssymb}
\usepackage{mathrsfs}
\usepackage{amsthm}
\usepackage{amsmath, amsrefs}
\usepackage[foot]{amsaddr}
\usepackage{mathtools}
\usepackage[cal=cm]{mathalfa}
\usepackage{stmaryrd}
\usepackage{upgreek}
\usepackage{bbm}
\usepackage{xypic, dynkin-diagrams}
\usepackage[hypertexnames=false, hyperfootnotes=false, colorlinks=true,linktocpage=true, allcolors=highlight, bookmarksdepth = 2]{hyperref}
\usepackage{cleveref}
\usepackage{url}
\usepackage{float}
\usepackage[full]{textcomp}
\makeatletter
\newcommand{\citecomment}[2][]{\citen{#2}#1\citevar}
\newcommand{\citeone}[1]{\citecomment{#1}}
\newcommand{\citetwo}[2][]{\citecomment[,~#1]{#2}}
\newcommand{\citevar}{\@ifnextchar\bgroup{;~\citeone}{\@ifnextchar[{;~\citetwo}{]}}}
\newcommand{\citefirst}{\@ifnextchar\bgroup{\citeone}{\@ifnextchar[{\citetwo}{]}}}

\makeatother

\usepackage{tikz,tikz-cd,tikz-3dplot}
\usepackage{pgfplots}
\pgfplotsset{compat=1.18}
\usetikzlibrary{calc}
\usetikzlibrary{fadings}
\usetikzlibrary{decorations.pathmorphing}
\usetikzlibrary{decorations.pathreplacing}
\usetikzlibrary{patterns}
\usetikzlibrary{arrows,shadows,positioning, calc, decorations.markings, 
	hobby,quotes,angles,decorations.pathreplacing,intersections,shapes}
\usetikzlibrary{fillbetween,backgrounds}
\usepackage{xcolor}

\definecolor{highlight}{HTML}{3465a4}

\usepackage{anyfontsize}
\usepackage[lining]{libertine}%
\usepackage{courier}
\usepackage{csquotes}
\makeatletter
\renewcommand*\libertine@figurestyle{LF}
\makeatother
\usepackage[libertine,libaltvw,liby]{newtxmath}
\usepackage{microtype}

\usepackage{array}
\newcolumntype{H}{>{\setbox0=\hbox\bgroup}c<{\egroup}@{}}

\BibSpec{article}{%
+{}{\sc \PrintAuthors} {author}
+{,}{ } {title}
+{,}{ \textit} {journal}
+{,}{ } {volume}
+{}{ \IfEmptyBibField{journal}{}{\PrintDate{date}}} {transition}
+{,}{ } {pages}
+{,}{ } {status}
+{.}{} {transition}
+{}{ \, Preprint available at \tt} {eprint}
+{.}{} {transition}
}

\BibSpec{book}{%
+{}{\sc \PrintAuthors} {author}
+{,}{ } {title}
+{.}{ \textit} {series}
+{,}{ } {volume}
+{.}{ } {publisher}
+{,}{ } {place}
+{,}{ } {date}
+{,}{ } {note}
+{.}{} {transition}
}

\BibSpec{incollection}{
+{} {\sc \PrintAuthors} {author}
+{,} { } {title}
+{,} { in \textit} {booktitle}
+{,}{ } {series}
+{,}{ } {volume}
+{}{ \PrintDate } {date}
+{,} { pp.~} {pages}
+{,}{ } {publisher}
+{ }{ Preprint available at \tt} {eprint}
+{.}{} {transition}
}

\BibSpec{inproceedings}{
+{} {\sc \PrintAuthors} {author}
+{,} { } {title}
+{,} { in \textit} {booktitle}
+{,} { pp.~} {pages}
+{,} { \PrintDateB} {date}
+{,} { available at \eprint} {eprint}
+{.} {} {transition}
}

\BibSpec{collection.article}{
+{} {\sc \PrintAuthors} {author}
+{,} { } {title}
+{,} { in \it \PrintConference} {conference}
+{,} { pp.~} {pages}
+{.} {\PrintBook} {book}
+{,} { \PrintDateB} {date}
+{.} {} {transition}
}

\BibSpec{innerbook}{
+{.} { \emph} {title}
+{.} { } {part}
+{:} { \emph} {subtitle}
+{.} { } {series}
+{,} { } {volume}
+{.} { Edited by \PrintNameList} {editor}
+{.} { Translated by \PrintNameList}{translator}
+{.} { \PrintContributions} {contribution}
+{.} { } {publisher}
+{.} { } {organization}
+{,} { } {address}
+{,} { \PrintEdition} {edition}
+{,} { \PrintDateB} {date}
+{.} { } {note}
}

\tikzset{
	commutative diagrams/.cd, 
	arrow style=tikz, 
	diagrams={>=stealth}
}
\tikzset{
	arrow/.pic={\path[tips,every arrow/.try,->,>=#1] (0,0) -- +(0,4pt);},
	pics/arrow/.default={triangle 90}
}
\tikzset{->-/.style={decoration={
			markings,
			mark=at position .6 with {\arrow{latex}}},postaction={decorate}}
}
\tikzset{
	c/.style={every coordinate/.try}
}
\usepackage{enumerate}

\usepackage{xifthen}
\newcommand{\Gm}[1][]{%
	\ifthenelse{\isempty{#1}}%
	{\mathbb{C}^\times}% if #1 is empty
	{(\mathbb{C}^\times)^{#1}}% if #1 is not empty
}

\newcommand{\re}{\mathrm{e}}

\newcommand{\de}{{\partial}}

\newcommand{\bbV}{\mathbb{V}}

\newcommand{\bbZ}{\mathbb{Z}}

\newcommand{\bbC}{\mathbb{C}}
\newcommand{\bbP}{\mathbb{P}}

\newcommand{\bbQ}{\mathbb{Q}}

\newcommand{\cT}{\mathcal{T}}

\newcommand{\cP}{\mathcal{P}}

\newcommand{\cW}{\mathcal{W}}

\newcommand{\cA}{\mathcal{A}}

\newcommand{\RR}{\mathcal{R}}

\newcommand{\cM}{\mathcal M}
\newcommand{\cQ}{\mathcal Q}

\renewcommand{\l}{\left}
\renewcommand{\r}{\right}

\def\beq{\begin{equation}}                     %  
	\def\eeq{\end{equation}}                       % 
\def\bea{\begin{eqnarray}}                     %         % 
	\def\eea{\end{eqnarray}}
\def\bary{\begin{array}} 
	\def\eary{\end{array}} 
\def\ben{\begin{enumerate}} 
	\def\een{\end{enumerate}}
\def\bit{\begin{itemize}} 
	\def\eit{\end{itemize}}

\theoremstyle{plain}
\newtheorem{thm}{Theorem}[section]
\newtheorem*{thm*}{Theorem}
\newtheorem{lem}[thm]{Lemma}

\newtheorem{prop}[thm]{Proposition}
\newtheorem*{prop*}{Proposition}

\newtheorem*{conj*}{Conjecture}

\newtheorem{cor}[thm]{Corollary}
\newtheorem*{cor*}{Corollary}
\theoremstyle{definition}
\newtheorem{defn}[thm]{Definition}
\newtheorem{rmk}[thm]{Remark}

\theoremstyle{plain}
\newtheorem{innercustomthm}{Theorem}
\newenvironment{customthm}[1]
{\renewcommand\theinnercustomthm{#1}\innercustomthm}
{\endinnercustomthm}

\theoremstyle{plain}

\theoremstyle{plain}

\theoremstyle{definition}

\theoremstyle{plain}

\crefname{equation}{Eq.}{Eqs.}
\crefname{eqnarray}{Eq.}{Eqs.}
\crefname{algo}{algorithm}{algorithms}
\crefname{conj}{conjecture}{conjectures}
\crefname{lem}{lemma}{lemmas}
\crefname{thm}{theorem}{theorems}
\crefname{claim}{claim}{claims}
\crefname{rmk}{remark}{remarks}
\crefname{prop}{proposition}{propositions}
\crefname{section}{section}{sections}
\crefname{appendix}{appendix}{appendices}
\crefname{cor}{corollary}{corollaries}
\crefname{figure}{figure}{figures}
\crefname{table}{table}{tables}
\crefname{example}{example}{examples}
\crefname{prob}{problem}{problems}
\crefname{assm}{assumption}{assumptions}
\crefname{defn}{definition}{definitions}
\crefname{notation}{notation}{notations}
\crefname{speculation}{speculation}{speculations}
\crefname{construction}{construction}{constructions}
\crefname{observation}{observation}{observations}
\crefname{innercustomthm}{Theorem}{Theorems}
\crefname{innercustomconj}{Conjecture}{Conjectures}
\crefname{innercustomassumption}{assumption}{Assumption}
\crefname{innerpracticalresult}{practical result}{practical results}

\newcommand{\bra}{\left\langle}
\newcommand{\ket}{\right\rangle}

\crefname{equation}{Eq.}{Eqs.}
\crefname{eqnarray}{Eq.}{Eqs.}
\crefname{algo}{Algorithm}{Algorithms}
\crefname{conj}{Conjecture}{Conjectures}
\crefname{lem}{Lemma}{Lemmas}
\crefname{thm}{Theorem}{Theorems}
\crefname{customthm}{Theorem}{Theorems}
\crefname{claim}{Claim}{Claims}
\crefname{rmk}{Remark}{Remarks}
\crefname{prop}{Proposition}{Propositions}
\crefname{section}{Section}{Sections}
\crefname{appendix}{Appendix}{Appendices}
\crefname{cor}{Corollary}{Corollaries}
\crefname{figure}{Figure}{Figures}
\crefname{table}{Table}{Tables}
\crefname{example}{Example}{Examples}
\crefname{prob}{Problem}{Problems}
\crefname{assm}{Assumption}{Assumptions}
\crefname{defn}{Definition}{Definitions}
\crefname{customconj}{Conjecture}{Conjectures}

\begin{document}

\title{
 The Saito Determinant for Extended Affine Weyl Discriminant Strata
  }

\begin{abstract}
We investigate the form of the Saito metric 
%of the canonical Frobenius manifold structure 
on quotients of the reflection representation of an extended affine Weyl group, restricted to an arbitrary stratum of its discriminant. We compute the Saito determinant for all strata and Dynkin types and show that it is divisible by a product of $q$-analogues of restricted roots for the associated hyperplane arrangement. Our results simultaneously generalise the factorisation in the Weyl denominator formula to higher codimension strata, and provide $q$-analogue refinements of results of Antoniou--Feigin--Strachan for Coxeter groups. 
\end{abstract}

\author{Andrea Brini}
\address{A.~Brini: School of Mathematical and Physical Sciences,
University of Sheffield, S3 7RH, Sheffield, United Kingdom. \newline On leave from CNRS, DR 13, Montpellier, France.
}
\email{a.brini@sheffield.ac.uk}
\author{Karoline van Gemst}
\address{K.~van Gemst: Dipartimento di Matematica e Applicazioni, Universit\`a di Milano-Bicocca, \newline
Via Roberto Cozzi 55, I-20125 Milano, Italy and INFN sezione di Milano-Bicocca}
\email{karoline.vangemst@unimib.it}

\maketitle

%This note assembles what is currently known about the determinant of the Saito metric restricted to discriminant strata of the Frobenius manifolds of Dubrovin–Zhang (DZ) type on orbit spaces of extended affine Weyl groups \cite{DZ1}. It is the analogue, in the trigonometric setting, of the results of Antoniou–Feigin–Strachan \cite{AntoniouFS:2020} for finite Coxeter groups. Two complementary strategies are presented:

%Part I (Sections 2–5) treats the classical types $A_l$, $B_l$, $C_l$, $D_l$ for arbitrary strata, using the Landau–Ginzburg superpotential description of the DZ Frobenius structures \cite{DZ1,DubrovinStrachanZhangZuo2019}. The superpotentials are meromorphic rather than polynomial, and the pole structure produces additional factors in the determinant, with exponents governed by the pole orders.

%Part II (Sections 6–10) treats strata of codimensions one, two and three uniformly in the Dynkin type (codimension three under the assumption that $\mathcal{R}$ is simply laced), by adapting the Jacobian-minor method of \cite[§6--7]{AntoniouFS:2020} to the trigonometric setting. The exposition of Part II is self-contained: all statements are proved from first principles, and the method of \cite{AntoniouFS:2020} is developed anew in the present context rather than cited.

\section{Introduction}

\subsection{Background}\label{sec:killing-normalization}

Let $\mathfrak{g}$ be a complex simple Lie algebra of rank $l$, and $(\mathfrak{h}, ())$ be its Cartan subalgebra equipped with the complexified Killing form. We will write $\mathcal{R}$ for the associated irreducible root system in $\mathfrak{h}^*$, and we shall further denote 
by $\Delta=\{\alpha_1,\dots,\alpha_l\}$ and $\Delta^*=\{\omega^\vee_1, \dots, \omega^\vee_l\}$ its set of simple roots and fundamental coweights respectively, $\mathcal{R}_+$ the set of positive roots, $\cA$ the associated hyperplane arrangement, $\mathcal{W}$ the Weyl group, and $h$ the Coxeter number. \\
%As in , for each $\RR$ w
As in \cite{DZ1,Brini:2025lax}, we will furthermore mark a special simple root $\alpha_{\bar k}$, ${\bar k} \in \llbracket 1,l\rrbracket$: when $\RR=A_l$, this can be any root, while for $\RR \neq A_l$, this is the root whose dual fundamental weight is the highest weight of the fundamental representation of $\mathfrak{g}$ of highest dimension. This choice determines an $l$-tuple of positive rational numbers given by
\[
d_i = (\omega_i, \omega_{\bar k}) \in \bbQ_{>0}\,, \quad i=1, \dots, l\,.
\]

To this data, Dubrovin--Zhang \cite{DZ1} associate an extension $\widetilde{\mathcal{W}}^{(\bar k)}$ of the affine Weyl group of type $\RR$,  acting on $\mathfrak{h}\oplus\mathbb{C}$,
%with linear coordinates $(x_1,\dots,x_l; x_{l+1})$ with respect to the basis of sim
%They 
and prove a Chevalley-type theorem for the ring of $\widetilde{\mathcal{W}}^{(\bar k)}$-invariant Fourier polynomials.
%, which is freely generated by $\{\tilde y_1,\dots,\tilde y_l,\tilde y_{l+1}\}$, where $\tilde y_1,\dots,\tilde y_l$ are, up to multiplication by powers of $w_0=\re^{c_\omega x_{l+1}}$, triangular combinations of the fundamental characters $(Y_1, \dots, Y_l)$ of the simple simply-connected group $\exp(\mathfrak{g})$, and $\tilde y_{l+1}$ is proportional to $x_{l+1}$: see \cite{DZ1,BvG,Brini:2025lax} for details. 
We depict the associated affine Dynkin diagrams, with a marking of the Dynkin node corresponding to $\alpha_{\bar k}$, in \cref{fig:dynkin}. The authors of \cite{DZ1} further show that the orbit space of $\widetilde{\cW}$ carries a canonical semisimple Frobenius structure $\mathcal{M}=\mathcal{M}(e,E,\eta,\gamma)$, where $\gamma$ is the  intersection form given by the direct sum of the Killing form on $\mathfrak{h}$ and a non-zero constant on the extended factor, with identity $e=\partial_{t_{\bar k}}$, and Euler field $E=d_{\bar k}^{-1}\de_{x_{l+1}}$. Here $(x_1, \dots, x_l; x_{l+1})$ denote linear coordinates on $\mathfrak{h}\oplus \bbC$ with respect to the basis $\Delta^* \cup \{v_{l+1}\}$, and $(t_1,\dots,t_{l+1})$ denote the flat coordinates of $\eta$, normalised so that
\begin{equation}\label{eq:killing-normalization}
\eta_{\alpha\beta}=\delta_{\alpha+\beta,\,l+2},
\qquad
t_{l+1}=\,x_{l+1}\,. 
%,\quad \in\mathbb{C}^{\times}.
\end{equation}
The Frobenius manifold $\cM$, along with its Dubrovin-dual \cite{MR2070050}, has been studied from a variety of angles in quantum cohomology \cite{Zaslow:1993fa, Brini:2013zsa,Brini:2014fea,Brini:2025lax,ma2026orbifold}, mirror symmetry \cite{DZ1,DubrovinStrachanZhangZuo2019, BvG}, trigonometric $\vee$-systems \cite{feigin2009trigonometric,Alkadhem:2020ybr}, integrable hierarchies \cite{Brini:2010ap,Brini:2011ff,Brini:2014mha,Milanov:2014pma}, and further relations to Chern--Simons \cite{Brini:2011wi,
Borot:2015fxa,Brini:2017tjw,Brini:2008ik} and Seiberg--Witten theory \cite{Brini:2009nbd,
Brini:2017gfi,Brini:2021wrm}, the latter being instrumental in a general construction of mirror Landau--Ginz\-burg superpotentials for $\cM$ \cite{Brini:2019agj,BvG, Brini:2017gfi}.

\begin{table}[t]
\def\arraystretch{1.5}    
    \begin{tabular}{|c|c|c|}
    \hline
    $\mathcal{R}$ & $\bar k$ & Canonically marked affine Dynkin diagram \\
    \hline \hline 
      $A_l$   & $1, \dots, l$ &  \begin{tikzpicture}[scale=3] 
  \dynkin[label, root radius=0.06cm, labels={\alpha_0,\alpha_1,\alpha_2,\alpha_3,\alpha_{\bar{k}}, \alpha_{l-2},\alpha_{l-1},\alpha_l},affine mark=*] A[1]{ooo.X.ooo}  
\end{tikzpicture} \\ \hline   
        $B_l$ &  $l-1$ &  \begin{tikzpicture}[scale=3]
  \dynkin[label,affine mark=*, root radius=0.06cm, labels={\alpha_0~,\alpha_1~,\alpha_2,\alpha_3,\alpha_j, \alpha_{l-2},\alpha_{l-1},\alpha_l}] B[1]{ooo.o.oXo}
\end{tikzpicture} \\ \hline 
$C_l$ &  $l$ &   \begin{tikzpicture}[scale=3]
  \dynkin[label,affine mark=*, root radius=0.06cm, labels={\alpha_0,\alpha_1,\alpha_2,\alpha_3,\alpha_j, \alpha_{l-2},\alpha_{l-1},\alpha_l}] C[1]{ooo.o.ooX}
\end{tikzpicture} \\ \hline 

$D_l$ &  $l-2$ &   \begin{tikzpicture}[scale=3]
  \dynkin[label,affine mark=*, root radius=0.06cm, labels={\alpha_0~,\alpha_1~,\alpha_2,\alpha_3,\alpha_j, \alpha_{l-2},~\alpha_{l-1},~\alpha_l}] D[1]{ooo.o.Xoo}
\end{tikzpicture} \\ \hline 

$E_6$ &   $3$   &   \begin{tikzpicture}[scale=3]
  \dynkin[label,affine mark=*, root radius=0.06cm, labels={~\alpha_0,\alpha_1,~\alpha_6,\alpha_2,\alpha_3, \alpha_{4},\alpha_{5}}] E[1]{oooXoo}
\end{tikzpicture} \\ \hline 

$E_7$
&   $3$ &      \begin{tikzpicture}[scale=3]
  \dynkin[label,affine mark=*, root radius=0.06cm, labels={\alpha_0,\alpha_1,~\alpha_7,\alpha_2,\alpha_3, \alpha_{4},\alpha_{5},\alpha_6}] E[1]{oooXooo}
\end{tikzpicture} \\ \hline 

$E_8$

&  $3$ &       \begin{tikzpicture}[scale=3]
  \dynkin[label,affine mark=*, root radius=0.06cm, labels={\alpha_0,\alpha_1,~\alpha_8,\alpha_2,\alpha_3, \alpha_{4},\alpha_{5},\alpha_6,\alpha_7}] E[1]{oooXoooo}
\end{tikzpicture} \\ \hline 

$F_4$

&   $2$ &  \begin{tikzpicture}[scale=3]
  \dynkin[label,affine mark=*, root radius=0.06cm, labels={\alpha_0,\alpha_1,\alpha_2,\alpha_3, \alpha_{4}}] F[1]{oXoo}
\end{tikzpicture} \\ \hline 

$G_2$
&  $2$ &   \begin{tikzpicture}[scale=3]
  \dynkin[label,affine mark=*, root radius=0.06cm, labels={\alpha_0,\alpha_1,\alpha_2}] G[1]{oX}
\end{tikzpicture} \\ \hline 

    \end{tabular}

\medskip 

\caption{Affine Dynkin diagrams with canonical markings, as in \cite[Table~1]{DZ1}. The node corresponding to the affine root is marked in black, and the canonical marked node is indicated with a $\times$.}

\label{fig:dynkin}
\end{table}

For $x\in \mathfrak{h}$, we shall write 
\[x_i = (\alpha_i, x)\,, \qquad \tilde x_i=(\omega_i^\vee,x)\,, \qquad  i=1,\dots,l\,,\] and let $\partial_{\alpha_j}$ denote the directional derivative along $\alpha_j$. Writing \[x=\sum_i\tilde x_i\alpha_i=\sum_i x_i \omega_i^\vee\,, \] 
%(which holds since $(\omega^\vee_j,\sum_i\tilde x_i\alpha_i)=\tilde x_j$), 
the coordinate vector fields of the system $(\tilde x_1,\dots,\tilde x_l)$ are
\begin{equation}\label{eq:tildex-coords}
\frac{\partial}{\partial\tilde x_i}=\partial_{\alpha_i}.
\end{equation}
%because $\partial_{\alpha_j}\tilde x_i=(\omega^\vee_i,\alpha_j)=\delta^i_j$. 
We shall similarly denote
%expanding $\omega^\vee_j=\sum_k(\omega^\vee_j,\omega^\vee_k)\alpha_k$ (pair both sides with $\omega^\vee_m$), 
\begin{equation}\label{eq:coweight-derivative}
\partial_{\omega^\vee_j}=\sum_{k=1}^{l}(\omega^\vee_j,\omega^\vee_k)\,\partial_{\alpha_k}
\end{equation}
for the derivative along a fundamental coweight. %We will henceforth use the coordinate system $z=(\tilde x_1,\dots,\tilde x_l,x_{l+1})$ on $V\oplus\mathbb{C}$.\\
In what follows, and throughout the paper, we shall use a symmetric normalisation for trigonometric forms on $\mathfrak{h}$: for any $\beta \in \RR$,
%of the weight lattice, 
we will write
\begin{equation*}
\mathsf{e}_\beta(x) := 2\sinh\!\big(\tfrac12\beta(x)\big),
\qquad
c_\beta(x) := \cosh\!\big(\tfrac12\beta(x)\big),
\end{equation*}
where $\beta(x)=(\beta,x)$. The trigonometric binomial $\mathsf{e}_\beta(x)$ can be regarded as a symmetric $q$-analogue of $\beta(x)$, upon sending $x \to \ln q\, x$, which gives $\mathsf{e}_\beta \to q^{x/2}-q^{-x/2}$.\\
We call a {\it unit} a  non-vanishing invertible element of the ring of Fourier polynomials in $(x_1, \dots, x_l; x_{l+1})/m$ for $m \in \bbZ$ -- that is, a non-zero constant multiple of $\re^{(\mu(x) + \nu x_{l+1})/m}$ for $\mu \in \Lambda_w$ a weight lattice element and $\nu \in \bbZ$.  Thus $\mathsf{e}_\beta=e^{-\beta/2}(e^\beta-1)$ differs from $e^{\beta(x)}-1$ by a unit, and for the fundamental coweights $\omega^\vee_i$, defined by $(\omega^\vee_i,\alpha_j)=\delta_{ij}$,
\begin{equation}\label{eq:trig-derivative}
\partial_{\omega^\vee_i}\mathsf{e}_\beta=\langle\beta,\omega^\vee_i\rangle\, c_\beta,
\qquad
c_\beta=1+O\big(\beta(x)^2\big),
\qquad
c_\beta\big|_{\beta(x)=0}=1 .
\end{equation}
Note that $c_\beta$ is even in $\beta(x)$ and equals $1$ to second order. We also record the additivity of restrictions: for weights $\beta,\gamma$,
\begin{equation}\label{eq:trig-additivity}
\mathsf{e}_{\beta+\gamma}\big|_{\{\gamma(x)=0\}} \;=\; \mathsf{e}_{\beta}\big|_{\{\gamma(x)=0\}} .
\end{equation}
%(powers of $w_0$ included). Proportionality statements ($\sim$ or $\propto$) are understood up to units. %Overall constants are recorded where they were computed but are immaterial to the structural statements.
In the following, for $p(x)$, $q(x)$ Fourier polynomials in $(x_1, \dots, x_{l+1})$, we shall write
\[
p \equiv q
\]
if $p=q$ modulo units. We shall also write $p\propto q$ if $p=c\,q$ for a non-zero constant $c$.\\

\subsection{The Saito determinant on discriminant strata}\label{sec:intro-strata} The determinant of the Gram matrix \eqref{eq:killing-normalization} of the Saito metric $\eta$ is obviously constant in the flat chart $(t_1, \dots, t_{l+1})$. 
From \cite{B02,Brini:2025lax}, the determinant of the Jacobian matrix associated to the change-of-variables $t_i \longrightarrow t_i(x_1, \dots , x_l, x_{l+1})$ is  
\beq 
J := \re^{\sum_{i=1}^l d_i x_{l+1}}\prod_{\beta\in R^+} \mathsf{e}_\beta(x)\,,
\label{disc}
\eeq
which is anti-invariant under the Weyl group action, 
    \[J(w(x)) = (-1)^{\sigma(w)} J(x).\]
%\end{thm}
The determinant of the Gram matrix of the Saito metric in the linear chart $(x_1, \dots, x_l;x_{l+1})$ therefore completely factorises as a product of linear trigonometric forms,
\beq
\det \eta(x) \propto J^2 \propto \re^{2x_{l+1}\sum_{i=1}^l d_i }\prod_{\beta \in \RR_+}\mathsf{e}^2_\beta(x)\,.
\label{eq:detbig}
\eeq 

By \eqref{eq:detbig}, the zero locus of $\det\eta(x)$ (i.e. the discriminant of $\cM$)
is the union of the affine mirrors $\{x:\beta(x)\in2\pi\mathrm{i}\,\bbZ\}$, $\beta\in\RR_+$,
of the reflections in $\widetilde{\mathcal{W}}^{(\bar k)}$, times the line parametrised by
$x_{l+1}$. 
%: the translation part of $\widetilde{\mathcal{W}}^{(\bar k)}$
%is $2\pi\mathrm{i}\,(Q^\vee+\bbZ\,\omega^\vee_{\bar k})$, and
%$\beta(Q^\vee+\bbZ\,\omega^\vee_{\bar k})=\bbZ$ for every root $\beta$ (for
%$\RR\neq C_l$ already $\beta(Q^\vee)=\bbZ$, while for the long roots of $C_l$ one has
%$\beta(Q^\vee)=2\bbZ$ and $\beta(\omega^\vee_l)=1$). 
Following the treatment of the
ordinary Coxeter case \cite{AntoniouFS:2020}, we are interested in calculating the
determinant of the pull-back of the Saito metric to an arbitrary stratum of the linear
arrangement $\cA$, i.e. a finite intersection of root hyperplanes through the origin.
Let \[\Pi_\beta=\{x\in V:\beta(x)=0\}\] be the mirror of $\beta \in \RR$, and for a set
$S=\{\beta_1,\dots,\beta_\kappa\}$ of linearly independent roots let
\[D=D_S=\bigcap_{\beta \in S}\Pi_{\beta}\,.\]
Note that each affine mirror is a $\widetilde{\mathcal{W}}^{(\bar k)}$-translate of the
linear mirror $\Pi_\beta$. The set $\RR_D=\RR\cap\langle S\rangle$ of roots vanishing on $D$ is a parabolic
subsystem of $\RR$, of which $D$ is the common kernel, and every parabolic subsystem is
$\cW$-conjugate to a standard one $\RR\cap\langle\Delta_I\rangle$, $I\subset\Delta$
\cite[Ch.~V, \S3.3]{B02}; since the metric data are
$\widetilde{\mathcal{W}}^{(\bar k)}$-invariant, it suffices to treat $S\subset\Delta$,
and we do so from now on.\footnote{Intersections of affine mirrors of codimension
greater than one need not be $\widetilde{\mathcal{W}}^{(\bar k)}$-conjugate to strata of
$\cA$: for instance $\{v_1=v_2=\pi\mathrm{i}\}$ in type $B_l$, along which the roots
vanishing modulo $2\pi\mathrm{i}$ form the non-parabolic subsystem $A_1\times A_1\subset B_2$.
In general the root subsystems attached to such strata are those of proper subdiagrams of
the affine Dynkin diagrams of \cref{fig:dynkin}, only some of which are conjugate to
subdiagrams of the finite ones. 
%For the classical types, the Landau--Ginzburg description of \cref{sec:Al-case,sec:BCD-case} is phrased in terms of coincidences on the Cartan torus and applies verbatim to all of them; for the exceptional types our results concern the strata of $\cA$ only.
} Define
\begin{equation}\label{eq:RD-decomposition}
\mathcal{R}_D=\mathcal{R}\cap\langle S\rangle
=\{\alpha\in\mathcal{R}:\alpha|_D=0\}\,.
%=\bigsqcup_{i=1}^{L}\mathcal{R}_D^{(i)}
\end{equation}
%(orthogonal decomposition into irreducible subsystems, of ranks $r_i$ and Coxeter numbers $h^{(i)}$), and, for $\beta\in\mathcal{R}\setminus\mathcal{R}_D$,
%\begin{equation}\label{eq:RDbeta-decomposition}
%\mathcal{R}_{D,\beta}=\langle\mathcal{R}_D,\beta\rangle\cap\mathcal{R}
%=\bigsqcup_{i=0}^{p}\mathcal{R}_{D,\beta}^{(i)},
%\qquad \beta\in\mathcal{R}_{D,\beta}^{(0)} .
%\end{equation}
Then $\beta\in\mathcal{R}\setminus\mathcal{R}_D$, $\beta|_D$ restricts to a non-zero linear form on $D$, and $\mathsf{e}_{\beta|_D}$ denotes the corresponding trigonometric form of Section~\ref{sec:killing-normalization}. \\

Let now $\eta_D \coloneqq i^*_D\eta$, with $i_D : D \hookrightarrow \cM$ the inclusion map. We shall prove an analogue of \eqref{eq:detbig} for a stratum $D$ of any codimension.
\begin{customthm}{A}
\label{thm:main}
For $\RR$ of any Dynkin type, and $D$ any discriminant stratum, the Gram matrix of the Saito metric in linear coordinates satisfies
\begin{equation}\label{eq:main-factorization}
\det\eta_D \;\equiv\; 
%\re^{h d x_{l+1}}
\prod_{\beta\in\mathcal{R}_+\setminus\mathcal{R}_D}\mathsf{e}_{\beta|_D}^{\,n_{\beta,D}},
\end{equation}
for $n_{\beta,D}\in \bbZ_{\geq 0}$. 
\end{customthm}

\begin{rmk}
    We note that the discriminant strata are natural submanifolds of the ambient Frobenius manifold as conjectured in \cite{Strachan:2004} (with $\eta_D$ non-degenerate on the generic locus, which is a consequence of the factorisation result of Theorem \ref{thm:main} together with an easy check for $\mathcal{R}$ exceptional and $\dim(D) = 1$). Indeed, the Euler vector field is tangent to $D$, while closure of the Frobenius multiplication on $TD$ follows from the restriction of the almost-dual structure; see  \cite[Chap.~7.2]{karoth} for details.
\end{rmk}

\begin{rmk}
In \cite{AntoniouFS:2020}, an analogous statement was made in the Coxeter case where the exponents were furthermore linked to  %$n_{\beta,D}$ — as a 
the Coxeter number of a root subsystem of $\RR$ specified by $\beta$ and $D$.
%, or otherwise — is asserted; indeed such an 
This identification fails in the extended affine setting, as explicit counterexamples show, where \eqref{eq:main-factorization} rather yields a non-trivial trigonometric refinement of the expressions of \cite{AntoniouFS:2020}. For example, if $D$ is a codimension-$(l-1)$ stratum parametrised by a linear coordinate $x_j$, \cite{AntoniouFS:2020} find that $\det \eta_D$ is a monomial $x_{j}^{h}$ where $h$ is the Coxeter number of $\RR$. Taking e.g. $\RR=E_8$ and $S=\Delta \setminus \{\alpha_3\}$, \cite[Thm.~7.1]{AntoniouFS:2020} gives in the polynomial Frobenius case
\[
\det \eta_D \propto x_3^{30}\,.
\]
The extended affine Weyl counterpart of that, as we shall see explicitly, is instead
\[
\det \eta_D \propto \re^{60 (x_9-x_3)} \left(1-\re^{x_3}\right) \left(1-\re^{2 x_3}\right){}^4 \left(1-\re^{3 x_3}\right){}^6 \left(1-\re^{4 x_3}\right){}^8 \left(1-\re^{5
   x_3}\right){}^5 \left(1-\re^{6 x_3}\right){}^6.
\]
Note that the exponents $n_{\beta_i,D}$ here
%coincide with the Coxeter label of $\beta_i=\alpha_i \in \Delta$ a simple root, and as such 
provide a partition of the Coxeter number $h=30$. Similarly, all other codimension $l-1$ strata will give  a refinement of the finite Coxeter setup, with different partitions of the Coxeter number.
\end{rmk}

To prove \cref{thm:main} we will adapt the strategy employed by \cite{AntoniouFS:2020} in the Coxeter case. When $\mathcal{R}\in\{A_l,B_l,C_l,D_l\}$ is classical, we will use the Landau--Ginzburg superpotential description of $\cM$ given in \cite{DZ1, DubrovinStrachanZhangZuo2019, BvG,Brini:2017gfi} to compute $n_{\beta,D}$ directly for every stratum. The result is given by Theorems \ref{thm:An-determinant} and \ref{thm:BCD-determinant}, and takes the form of a finite-type multiplicity reduced, at $\beta$ whose restriction to $D$ coincides with a pole locus of the Landau–Ginzburg superpotential, by the order of that pole.
Beyond the classical types we will establish the factorisation \eqref{eq:main-factorization} itself, together with explicit formulas for $\det\eta_D$ in terms of Jacobian minors, for strata of codimension at most two (codimension three when $\RR$ is simply laced). This leaves out the case of codimension three strata for $\RR=F_4$, and codimension greater than three strata for $E_l$: because these strata are of relatively low dimension, they are amenable to a direct calculation which we systematically perform. The results for these exceptional cases are available in the ancillary {\it Wolfram Language} notebook.

%\textbf{Notation.} In Part I, $\lambda$ denotes the superpotential and $\mu$ (respectively $P$) its argument; $k$ abbreviates the marked-node label $\bar k$; $d+1=\dim D$; $m_a$ are stratum multiplicities, $\xi_a$ stratum coordinates, $q_i$ critical points of $\lambda_D$, and $u_i=\lambda_D(q_i)$ canonical coordinates. 
%In Part II, following \cite[§7.4]{AntoniouFS:2020}, the letters $\lambda,\nu,\theta$ denote the three simple roots of a codimension-three stratum; superpotentials do not occur there. 
%In Section 9 we write $d_S=|\mathcal{R}_+\cap\langle S\rangle|-2$, reserving $d$ for dimensions.

\subsection*{Acknowledgements} We  acknowledge HPC facilities at the University of Sheffield through its {\it Stanage} compute cluster, which we employed for the calculation of the Saito determinants of the  codimension four strata for $\RR=E_8$. A.~B~ was supported by EPSRC through grants refs.~EP/S003657/2 and UKRI2394. K. v. G. was supported by EPSRC through grant ref.~EP/S003657/2 and by INFN through the IS-CSN4 ``Mathematical methods of
nonlinear physics''. 

\section{Classical series}

The ring of $\cW$-invariant regular functions on the Cartan torus $\cT\coloneq \exp(\mathfrak{h})$ is generated by the Weyl orbit sums of the fundamental weights,
\[
Y_i(x_1, \dots, x_l) = \sum_{w \in \cW} \re^{\bra w(\omega_i), x\ket}\,.
\]
As in \cite{DZ1}, for an element of $\mathfrak{h}\oplus \bbC$, we will write
\[
y_i(x_1, \dots, x_l; x_{l+1}) = 
\begin{cases}
    \re^{d_i x_{l+1}}Y_i(x_1, \dots, x_{l})\, & i\leq l\,, \\
    x_{l+1} & i=l+1\,.
\end{cases} 
\]
By \cite{BvG}, $\cM$ admits a Landau–Ginzburg description in terms of a superpotential/primitive form pair on a family of algebraic curves \cite{BvG}.  Fix  a non-zero dominant weight $\omega $. Starting from the characteristic polynomial of a regular element in the representation $\rho_\omega \in \mathrm{Rep}(\cT)$,
\beq
\cQ = \prod_{\omega'\in \cW(\omega)} \l(\re^{\bra\omega', x\ket }-\mu\r) \in \bbQ[Y_1, \dots , Y_l][\mu]\,,
\label{eq:weylrel}
\eeq
we define the two-variable polynomial
\beq 
\cP\l(y_1, \dots , y_{l+1};\lambda,\mu\r) \coloneqq
 \cQ\l(Y_i=y_i\re^{-d_i x_{l+1}}-\delta_{i \bar k} \lambda \re^{-d_{\bar k}x_{l+1}};\mu \r)\,.
 \label{eq:shiftfun}
\eeq 
%
%where $\omega_{\bar k} \coloneqq \widehat{\omega}$, and consider, 
%for $i=1, \dots , l_{\RR}$ and 
Over the complement of the discriminant, \eqref{eq:shiftfun} defines a plane algebraic curve 
%$\bbA^2_{\lambda,\mu}$ with fibre at $w$ given by
%\beq
%\[ 
$C_y \coloneqq \bbV\l(\cP\r)\,.$
%\]
%\eeq
%
Let $\overline{C_y}$ denote the normalisation of the projective closure of the fibre at $y$,
%$g\coloneqq h^{1,0}\big(\overline{C_y}\big)$, and let $\mathsf{m}$ be the ramification profile over infinity of the 
%restriction of the 
%Cartesian projection $\lambda :  \overline{C_y} \rightarrow \bbP^1$. The corresponding family is the pull-back of the universal curve to $M^{\rm ss}_{\rm AW}$, where the pull-back  metric, product, and intersection tensors are given by \eqref{eq:etares}--\eqref{eq:gres}.
%subvariety of  % subvariety 
%isomorphic to torus $\bbC^\times \times (\bbC^\times)^{l_{\RR}}$ 
%parametrised with coordinates , 
% $M_\omega_{\rm LG}$ of the Hurwitz space $H_{g_{\omega},\mathsf{m}_{\omega}}$, where $g_\omega=h^{1,0}\big(\overline{C_w}\big)$, and $\mathsf{m}_{\omega}$ records the ramification profile at infinity of the $\lambda$-projection.
 %
% The definition of $C_w$ does not immediately clarify that the ramification profile $\mathsf{m}_\omega$ is moduli-independent. 
%
%\bea
%\label{eq:todaeta}
%\eta_{\rm LG}(\de_{y_i}, \de_{y_j}) &=&   \sum_{m}\underset{p_m^{\text{cr}}}{\text{Res}}\frac{\delta_{\de_{y_i}}\lambda~
%  \delta_{\de_{y_j}}\lambda}{\rd \lambda}\phi^2,  \\
%\label{eq:todac}
%c_{\rm LG}(\de_{y_i}, \de_{y_j}, \de_{y_k}) &=&  
%\sum_{m}\underset{p_m^{\text{cr}}}{\text{Res}}\frac{\delta_{\de_{y_i}}\lambda~
%  \delta_{\de_{y_j}}\lambda \delta_{\de_{y_k}}\lambda}{\rd \lambda}\phi^2,\\
%\label{eq:todag}
%\eta^\flat_{\rm LG}(\de_{y_i}, \de_{y_j}) &=&  \sum_{m}\underset{p_m^{\text{cr}}}{\text{Res}}\frac{\delta_{\de_{y_i}}\log \lambda~
%  \delta_{\de_{y_j}} \log \lambda}{\rd \log \lambda}\phi^2, 
%\eea
%
with $\{q_i\}_i$ denoting the ramification points of $\lambda :
\overline{C_y} \longrightarrow \bbP^1$. When $\RR \in \{A_l,B_l,C_l,D_l\}$ and $\omega$ is the highest weight of the\footnote{This choice isn't always unique, but different choices are related by a change of basis in $\mathfrak{h}$ induced by an outer automorphism of the Lie algebra. More specifically, for $\RR=A_l$, we have two minimal choices related by duality; and for $\RR=D_4$ we have three choices related by triality. In each case, the relation between them is a linear change of coordinates in the $x$-frame.}  fundamental representation of minimal dimension, we have from \cite{BvG} that $\deg_\lambda \cP = 1$, hence $h^{1,0}(C_y)=0$, and $\overline{C}_y\simeq \bbP^1$ over the smooth locus. The Landau--Ginzburg superpotential defined by the $\lambda$-projection is then a  rational function \cite{BvG,DZ1,DubrovinStrachanZhangZuo2019} on $\bbP^1$, and for tangential vector fields $\zeta_i$ on $D\setminus\Sigma_D$, the Landau--Ginzburg formula for the pull-back $\eta_D$ of the Saito metric to $D$ 
%Writing fundamental characters in the $x$-frame via the Weyl character formula, \[\tilde y_j=\sum_{\omega'}\mathfrak{m}'\,\re^{(\omega',x)}\] for $j\neq0$ and $w_0=\re^{c_\omega x_{l+1}}$, the superpotential becomes a rational function of $\mu$ (type $A$) or of $P$ (types $B,C,D$) whose zeros are exponentials of linear forms in $x$. The Saito metric and the structure constants on a stratum $D$ are computed by the residue formulas
\begin{equation}\label{eq:residue-formulas}
\eta_D(\zeta_1,\zeta_2)=\epsilon_{\RR}\!\!\sum_{q_s\in\mathrm{Crit}(\lambda_D)}\!\!\operatorname*{Res}_{\,q_s}\frac{\zeta_1(\lambda_D)\,\zeta_2(\lambda_D)}{\phi\;\lambda_D'}\,\mathrm{d}\varsigma,
%\qquad
%c_D(\zeta_1,\zeta_2,\zeta_3)=\!\!\sum_{q_s}\!\operatorname*{Res}_{\,q_s}\frac{\zeta_1(\lambda_D)\,\zeta_2(\lambda_D)\,\zeta_3(\lambda_D)}{\phi\;\lambda_D'}\,\mathrm{d}\varsigma,
\end{equation}
where $(\varsigma,\phi,\epsilon_{\RR})=(\mu,\,\mu^2,\,(-1)^{\bar k+1})$ in type $A$, and $(\varsigma,\phi,\epsilon_{\RR})=(P,\,P^2-1,\,(-1)^{\bar{k}})$ in types $B$, $C$, $D$, with the variable $P$ and the exponent $\bar{k}$ as in \cref{sec:BCD-case}. The overall normalisations are fixed to match the conventions of \cite[Chap.~7.3]{karoth}.  By turning the contour around, \eqref{eq:residue-formulas}
reduces to a calculation of residues at the poles of the rational function $\lambda_D \coloneqq \lambda|_D$. 
%In \eqref{eq:residue-formulas}, $\lambda_D:=\lambda|_D$, and we have  $(\varsigma,\phi)=(\mu,\mu^2)$ in type $A$ and $(\varsigma,\phi)=(P,\,P^2-1)$ in types $B,C,D$, up to the sign conventions recorded below. 

The discriminant is the locus where $\lambda$ vanishes at one of its critical points, that is, where a zero of $\lambda$ collides with another zero or, in the $BCD$ case, with a pole. We will, accordingly, label strata by the resulting coincidence patterns. We shall be brief on the details of some of the more involved (if entirely straightforward) residue calculations; the curious reader is referred to \cite[Chap.~7.3]{karoth} for a fuller account. 

\subsection{Case I: $\RR=A_l$}\label{sec:Al-case}

%\subsubsection{Superpotential and strata}

Fixing a marked node $\bar k \in \llbracket 1,l\rrbracket$, we have \cite{DZ1}
\begin{equation}\label{eq:An-superpotential}
\lambda(\hat{\mu})=\frac{(-1)^{\bar k+l}\,a_0\,\prod_{i=1}^{l+1}(\hat{\mu}-a_i)}{\hat{\mu}^{\bar k}}\,,
\end{equation}
where
\[
a_0=\re^{d_{\bar k} x_{l+1}}\,, \quad a_1=\re^{-x_1}\,,\quad a_{l+1}=\re^{x_l}\,,\quad a_i=\re^{x_{i-1}-x_i}\, \text{otherwise}.
\]
%so that $\prod_i a_i=1$.
%The normalisation \[ a_i=a_i\re^{c_\omega x_{l+1}/(l+1-\bar k)}, \quad \mu=\hat{\mu}\,\re^{-c_\omega x_{l+1}/(l+1-\bar k)}\] absorbs $w_0$ into the product.
%; hats are dropped until the change of frame at the end of the section.
Vanishing of $\lambda$ at a critical point forces $a_i=a_j$ for some $i\neq j$; a stratum is therefore a coincidence pattern
\begin{equation*}
a_{1}=\dots=a_{m_1}=\hat{\xi}_1,\ \dots,\
a_{\,\sum_{j<d+1}m_j+1}=\dots=a_{\,\sum_{j\le d+1}m_j}=\hat{\xi}_{d+1},
\qquad \sum_{a=1}^{d+1}m_a=l+1,
\end{equation*}
with $\prod_a\hat{\xi}_a^{m_a}=1$, and
\begin{equation}\label{eq:An-stratum-superpotential}
\lambda_D(\hat{\mu})=\frac{(-1)^{l+\bar k} a_0}{\hat{\mu}^{\bar k}}\prod_{a=1}^{d+1}(\hat{\mu}-\hat{\xi}_a)^{m_a}.
\end{equation}

Define $s \coloneqq a_0^{\frac{1}{l+1-\bar{k}}}$ and rescale 
\begin{equation}
    \mu = s\hat{\mu}, \quad \xi_a = s\hat{\xi}_a, \, \forall \, a=1, \cdots, d+1.
\end{equation}
Then, $\lambda$ carries no explicit $a_0$ prefactor 
\begin{equation}\label{eq:An-stratum-superpotentialScaled}
    \lambda_D(\mu) = \dfrac{(-1)^{l+\bar{k}}}{\mu^{\bar{k}}}\prod_{a=1}^{d+1}(\mu-\xi_a)^{m_a},
\end{equation}
and the coincidence values $\xi_a$, $a=1,\dots,d+1$, are functionally independent coordinates satisfying $\prod_a \xi_a^{m_a} = a_0^{\frac{l+1}{l+1-\bar{k}}}$. We shall work with the scaled quantities $(\mu;\underline{\xi})$ until Theorem \ref{thm:An-determinant}. 
%\subsubsection{Critical data and canonical coordinates}

\begin{lem}\label{lem:An-critical-derivative}
We have
\[
\lambda_D'(\mu)=\frac{(-1)^{l+\bar k}(l+1-\bar k)}{\mu^{\bar k+1}}\prod_{a=1}^{d+1}(\mu-\xi_a)^{m_a-1}\prod_{j=0}^{d}(\mu-q_j)\] for some $q_j\in\mathbb{C}$ and
\begin{equation*}
\lambda_D''(q_r)=\frac{(-1)^{l+\bar k}(l+1-\bar k)}{q_r^{\,\bar k+1}}\prod_{a=1}^{d+1}(q_r-\xi_a)^{m_a-1}\prod_{j\neq r}(q_r-q_j).
\end{equation*}
\end{lem}

\begin{proof}
Differentiation of \eqref{eq:An-stratum-superpotentialScaled} yields \[\lambda_D'=\frac{(-1)^{l+\bar k}}{\mu^{\bar k+1}}\prod_a(\mu-\xi_a)^{m_a-1}R(\mu)\] with $R(\mu)$ a polynomial of degree $d+1$ and leading coefficient $l+1-\bar k$. Factorising $R$ over its roots gives the first formula. Differentiating once more and evaluating at $q_r$, only the term in which the derivative falls on $\prod_j(\mu-q_j)$ survives.
\end{proof}

We will call the critical values $u_i=\lambda_D(q_i)$, $i=0,\dots,d$ the {\it canonical coordinates} on the stratum.
%Since $\lambda_D(\mu)-\lambda_D(q_j)$ vanishes to second order at $\mu=q_j$, differentiation in $u_i$ gives $\partial_{u_i}\lambda_D(\mu)|_{\mu=q_j}=\partial_{u_i}u_j=\delta_{ij}$.

\begin{prop}\label{prop:An-derivative-formulas}
The fixed $\mu$-derivative of the superpotential and its zeroes with respect to the canonical coordinates on the stratum are
\[
\partial_{u_r}\lambda_D(\mu)=\frac{\mu}{q_r(\mu-q_r)}\,\frac{\lambda_D'(\mu)}{\lambda_D''(q_r)},
\qquad
\partial_{u_r}\xi_a=\frac{\xi_a}{q_r(q_r-\xi_a)\,\lambda_D''(q_r)}.
\]
\end{prop}

\begin{proof}
From \eqref{eq:An-stratum-superpotentialScaled}, \[\partial_{u_r}\lambda_D=\frac{(-1)^{l+\bar k}}{\mu^{\bar k}}\prod_a(\mu-\xi_a)^{m_a-1}F_r(\mu)\] with $\deg_\mu F_r(\mu)=d$. The $d+1$ values $F_r(q_i)$ are fixed by $\partial_{u_r}\lambda_D(q_i)=\delta_{ri}$, and Lagrange interpolation reconstructs $F_r$; comparison with Lemma~\ref{lem:An-critical-derivative} gives the first formula. For the second, differentiate \eqref{eq:An-stratum-superpotentialScaled} logarithmically in $u_r$, divide by $(\mu-\xi_a)^{m_a-1}$, let $\mu\to\xi_a$, and use \[\frac{\lambda_D}{(\mu-\xi_a)^{m_a}}\bigg|_{\xi_a}=\frac{\lambda_D'}{(m_a(\mu-\xi_a)^{m_a-1})}\bigg|_{\xi_a}.\]
\end{proof}

\begin{lem}\label{lem:An-canonical-coords}
In canonical coordinates for $D$, we have
\[
\eta_D(\partial_{u_i},\partial_{u_j})=\frac{(-1)^{\bar k+1}\delta_{ij}}{q_i^2\,\lambda_D''(q_i)}.
\]
%\qquad
%\partial_{u_i}\cdot\partial_{u_j}=\delta_{ij}\,\partial_{u_j}.
Furthermore, the following equalities hold:

\beq
\prod_{i=0}^{d}q_i=-\frac{\bar k}{\,l+1-\bar k\,}\prod_{a=1}^{d+1}\xi_a\,,
\qquad
\prod_{i=0}^{d}(\xi_a-q_i)=\frac{\xi_a\,m_a}{\,l+1-\bar k\,}\prod_{i\neq a}(\xi_a-\xi_i)\,,
\label{eq:An-identities}
\eeq
\beq
\frac{\prod_{i=0}^{d}\lambda_D''(q_i)}{\prod_{i<j}(q_i-q_j)^2}
= \, \frac{\prod_{i,a}(q_i-\xi_a)^{m_a-1}}{\prod_{a}\xi_a^{\,\bar k+1}}\,
\frac{(l+1-\bar k)^{d+\bar k+2}\,(-1)^{(l+\bar k)d+\frac{d(d+1)}{2}+l+1}}{\bar k^{\,\bar k+1}}.
\label{eq:z-identity}
\eeq
\end{lem}
\begin{proof}
Substituting Proposition~\ref{prop:An-derivative-formulas} into \eqref{eq:residue-formulas}, the integrand acquires the factor $\mu^2\lambda_D'(\mu)/((\mu-q_i)(\mu-q_j))$ divided by $\mu^2\lambda_D'$; by Lemma~\ref{lem:An-critical-derivative} the residues vanish unless $i=j$, and for $i=j$ the residue at $q_i$ equals $1/(q_i^2\lambda_D''(q_i))$, so the prefactor $\epsilon_{\RR}=(-1)^{\bar k+1}$ in \eqref{eq:residue-formulas} gives the stated expression. Equating the expression of Lemma~\ref{lem:An-critical-derivative} with the direct differentiation of \eqref{eq:An-stratum-superpotentialScaled}, and comparing constant terms, gives the first identity in \eqref{eq:An-identities}, while evaluation of \[\frac{\mu^{\bar k+1}\lambda_D'}{\big((l+1-\bar k)\prod_i(\mu-\xi_i)^{m_i-1}\big)}\bigg|_{\mu=\xi_a}\,,\] using the limit relation quoted in the proof of Proposition~\ref{prop:An-derivative-formulas}, gives the second. Finally, substitute the second formula of Lemma~\ref{lem:An-critical-derivative} for each factor $\lambda_D''(q_i)$; the products $\prod_{j\neq i}(q_i-q_j)$ combine with the denominator to give a sign, and the product $\prod_i q_i^{-(\bar k+1)}$ is evaluated by the first identity in \eqref{eq:An-identities}.
\end{proof}

%subsubsection{The determinant}

\begin{thm}\label{thm:An-determinant}
In the frame of linear coordinates on $D$ we have
\begin{equation}
\det\eta_D\,\propto\,a_0^{\,d+1}\prod_{a=1}^{d+1}\hat{\xi}_a^{\;-\bar k}
\prod_{1\le a<b\le d+1}(\hat{\xi}_a-\hat{\xi}_b)^{\,m_a+m_b}\,,
\label{eq:An-determinant-xframe}
\end{equation}
with a proportionality constant depending on the stratum.
\end{thm}

\begin{proof}
Using Lemmas \ref{lem:An-critical-derivative}--\ref{lem:An-canonical-coords}  gives \[\det\eta_D(\xi)=(\det B)^{-2}\det\eta_D(u)\] with $B=(\partial_{u_r}\xi_a)$. By Proposition~\ref{prop:An-derivative-formulas}, $B$ is a Cauchy matrix in the variables $(\xi_a,q_r)$ up to row and column rescalings, so its determinant is expressed by the Cauchy formula, and combining Lemma \ref{lem:An-canonical-coords} with the identity \[\prod_a\prod_{b\neq a}(\xi_a-\xi_b)^{m_a+1}=(-1)^{\sum_a (a-1)\,m_a-\frac{d(d+1)}{2}}\prod_{a<b}(\xi_a-\xi_b)^{m_a+m_b+2}\] yields \[\det\eta_D(\xi)\,\propto\,\prod_{a}\xi_a^{\;m_a-\bar k-2}\prod_{a<b}(\xi_a-\xi_b)^{\,m_a+m_b}\,.\] Each $\xi_a$ is a non-zero constant multiple of the exponential of a linear form on $D$, so the Jacobian matrix of $(\xi_1,\dots,\xi_{d+1})$ with respect to any frame of linear coordinates on $D$ has $a$-th row proportional to $\xi_a$, with constant coefficients; since the $\xi_a$ are independent, its determinant is a non-zero constant multiple of $\prod_a\xi_a$. The change to linear coordinates therefore multiplies the determinant by $\big(\prod_a\xi_a\big)^{2}$ up to a constant, shifting the exponents to $m_a-\bar k$.  Finally, substituting $\xi_a=\hat\xi_a\,a_0^{\frac{1}{l+1-\bar k}}$ and using $\prod_a\hat{\xi}_a^{m_a}=1$ together with $\sum_{a<b}(m_a+m_b)=d(l+1)$, the powers of $a_0$ assemble to $d+1$ and the assertion follows.
\end{proof}
Each factor $\hat{\xi}_a-\hat{\xi}_b$ is a unit multiple of $\mathsf{e}_{\beta|_D}$ for the corresponding root $\beta$ (a difference of two of the entries $a_i$), with exponent $n_{\beta,D}=m_a+m_b$. 
%(which for $A_l$ happens to equal $h(A_{m_a+m_b-1})$, a coincidence special to this type, not asserted in general). 
The factors $\hat{\xi}_a^{-\bar k}$ and the power of $a_0$ are units: the single finite pole of \eqref{eq:An-stratum-superpotential} is at $\hat{\mu}=0$, of order $\bar k$, away from all zeros. We reach the following conclusion.

\begin{cor}\label{cor:An-main-thm}
Theorem~\ref{thm:main} holds for $\RR=A_l$.
\end{cor}

%\begin{example}[$A_3$, $\bar k=1$]\label{ex:A3-k1}
%For the stratum $a_3\mapsto a_1$ of $\lambda=a_0\,\mu^{-1}\prod_{i=1}^4(\mu-a_i)$, that is \[\lambda_D=\frac{\re^{x_{l+1}/2}(\mu-\re^{-x_1})^2(\mu-\re^{x_1-x_2})(\mu-\re^{x_1+x_2})}{\mu}\,,\] the residue formula \eqref{eq:residue-formulas} gives
%\begin{equation*}
%\det\eta_D(x)=-\tfrac{32}{81}\,(\re^{2x_2}-1)^2\,(\re^{2x_1+x_2}-1)^3\,(\re^{2x_1-x_2}-1)^3\,\re^{3 x_{l+1}/2-5x_1-2x_2},
%\end{equation*}
%in agreement with \eqref{eq:An-determinant-xframe} for $(m_1,m_2,m_3)=(2,1,1)$, $d=2$. 
%the multiplicities $3,3,2$ are $h(A_2),h(A_2),h(A_1)$, and the exponential prefactor is a unit.
%\end{example}

\subsection{Case II: $\RR\in\{B_l,C_l,D_l\}$}\label{sec:BCD-case}

\subsubsection{Superpotentials}

In the $x$-frame,
\begin{equation}\label{eq:BCD-superpotential-mu}
\lambda(\mu)=\frac{(-1)^{\bar{k}+1}\,\re^{d_{\bar k} x_{l+1}}}{\mu^{\bar{k}}(\mu+1)^{\ell_1}(\mu-1)^{\ell_2}}
\prod_{j=1}^{l}(\mu-a_j)(\mu-a_j^{-1}),
\end{equation}
with $(\bar{k},\ell_1,\ell_2)=(l-1,2,0)$, $(l,0,0)$, $(l-2,2,2)$ for $B_l$, $C_l$, $D_l$ respectively, 
%(canonical markings; non-canonical markings correspond to general values of these parameters), 
and $a_1=\re^{x_1}$, $a_i=\re^{x_i-x_{i-1}}$ generically, with the type-dependent modifications at the branch node. 
%The case $\RR=G_2$ is obtained from $\RR_=D_4$ by restricting to the triality-invariant locus $x_1=x_3=x_4$. 
\\

The substitution $\mu=\re^{2i\phi}$, $P=\cos\phi$ brings \eqref{eq:BCD-superpotential-mu} to the form of \cite{DubrovinStrachanZhangZuo2019},
\begin{equation}\label{eq:BCD-superpotential-P}
\lambda(P)=\frac{a_0}{(P^2-1)^{M}\,P^{2N_0}}\prod_{j=1}^{l}\big(P^2-p_j^2\big),
\qquad
p_j^2=\frac{a_j+a_j^{-1}+2}{4},
\quad l=\bar{k}+M+N_0,
\end{equation}
with $\ell_1=2N_0$, $\ell_2=2M$, and $a_0=4^{\bar{k}}(-1)^{\bar{k}+1}\re^{d_{\bar k} x_{l+1}}$. Since $p_j=\cosh(v_j/2)$ for $a_j=\re^{v_j}$, one has that $p_a^2-p_b^2$ is a unit multiple of $\mathsf{e}_{v_a-v_b}\mathsf{e}_{v_a+v_b}$ and $p_a^2-1$ a unit multiple of $\mathsf{e}_{v_a}^{2}$; this is the correspondence between the Landau–Ginzburg variables and the trigonometric forms of \cref{sec:killing-normalization}.
\subsubsection{Strata}

Vanishing of $\lambda$ at a critical point forces either $p_a= \pm p_i$ for some $i\neq a$, or $p_a=0$, or $p_a^2=1$, the last two meaning that a zero of $\lambda$ sits at a branch point of $\mu\mapsto P^2$. A stratum is therefore a pattern
\begin{equation*}
p_1=\dots=p_{m_0}=0,\qquad p_{m_0+1}^2=\dots=p_{m_0+n}^2=1,\qquad
\epsilon_\bullet p_{\bullet}=\xi_1\ (m_1\ \text{equations}),\ \dots,\
\epsilon_\bullet p_{\bullet}=\xi_d\ (m_d\ \text{equations}),
\end{equation*}
with $\epsilon_\bullet\in\{\pm1\}$, $m_0,n\ge0$, $\sum_{j=1}^{d}m_j=l-m_0-n$; $\dim D=d+1$ with coordinates $(a_0,\xi_1,\dots,\xi_d)$, and
\begin{equation}\label{eq:BCD-stratum-superpotential}
\lambda_D(P)=\frac{a_0}{(P^2-1)^{M-n}\,P^{2(N_0-m_0)}}\prod_{i=1}^{d}\big(P^2-\xi_i^2\big)^{m_i}.
\end{equation}
The locus so defined is a stratum of the discriminant, i.e. an intersection of $l-d$ independent mirrors, if and only if $\lambda_D$ vanishes at every coincidence, i.e. $m_0\in\{0\}\cup(N_0,l]$ and $n\in\{0\}\cup(M,l]$: a single zero colliding with a pole of $\lambda$ is not a discriminant condition, consistently with \eqref{eq:detbig}. Since $p_a^2=1$ means $v_a=0$, the strata $D_S$ with $S\subset\Delta$ of \cref{sec:intro-strata} have $m_0=0$, with $n>0$ exactly when $S\ni\alpha_l$ in types $B_l$, $C_l$ (short, resp.\ long, simple root) or $S\supseteq\{\alpha_{l-1},\alpha_l\}$ in type $D_l$; in type $C_l$, the translation $x\mapsto x+2\pi\mathrm{i}\,\omega^\vee_l$ in $\widetilde{\cW}^{(\bar k)}$ acts as $P^2\mapsto 1-P^2$ and exchanges the blocks at $P=0$ and $P^2=1$. It is convenient to treat all cases at once.

%\subsubsection{Critical data, canonical coordinates, product identities}

\begin{lem}\label{lem:BCD-critical-derivative}
There are $q_0,\dots,q_d\in\bbC$ such that
$$
\lambda_D'(P)=\frac{2\bar{k}\,a_0}{(P^2-1)^{M-n+1}P^{2(N_0-m_0)+1}}
\prod_{i=1}^{d}(P^2-\xi_i^2)^{m_i-1}\prod_{i=0}^{d}(P^2-q_i^2),
$$
and, writing
\begin{equation*}
\kappa_r\coloneqq\frac{4\bar{k}\,a_0\,(q_r^2-1)^{n-M}}{q_r^{2(N_0-m_0)}}
\prod_{i=1}^{d}(q_r^2-\xi_i^2)^{m_i-1}\prod_{i\neq r}(q_r^2-q_i^2),
\end{equation*}
one has $\kappa_r=(q_r^2-1)\lambda_D''(q_r)$ whenever $q_r\notin\{0,\pm1\}$. Generically on $D$, exactly one $q_r$ vanishes if $m_0=N_0$, and exactly one satisfies $q_r^2=1$ if $n=M$; in these two cases $\kappa_r=-2\lambda_D''(0)$, respectively $\kappa_r=2\lambda_D'(1)$.
\end{lem}

\begin{proof}
Logarithmic differentiation of \eqref{eq:BCD-stratum-superpotential} gives \[\lambda_D'(P)=2P\lambda_D\bigg[(m_0-N_0)P^{-2}+(n-M)(P^2-1)^{-1}+\sum_i m_i(P^2-\xi_i^2)^{-1}\bigg]\,.\] Over the common denominator the bracket has numerator an even polynomial of degree $2(d+1)$ in $P$ with leading coefficient $m_0-N_0+n-M+\sum_i m_i=\bar k$, whose roots are the $\pm q_i$. This is an identity of rational functions: if $m_0=N_0$ (resp.\ $n=M$) the factor $P^{-1}$ (resp.\ $(P^2-1)^{-1}$) is cancelled by one factor of the last product, which is the statement about vanishing $q_r$ (resp.\ $q_r^2=1$). The relations between $\kappa_r$ and the derivatives of $\lambda_D$ follow by differentiating the displayed formula once more at a simple zero $q_r\neq0$ of $\lambda_D'$, or once at $P=0$, or by evaluating it at $P=1$.
\end{proof}

We call $q_0,\dots,q_d$ the \emph{nodes} of $\lambda_D$: they are the critical points of $\lambda_D$ on $\overline{C}_y$ in the sense of \eqref{eq:residue-formulas}, the node at $P=0$ (resp.\ $P^2=1$) being the branch point $\mu=-1$ (resp.\ $\mu=1$) when $\lambda_D$ is regular there. With $u_i=\lambda_D(q_i)$, $i=0,\dots,d$, one has $\partial_{u_i}\lambda_D(P)|_{P=q_j}=\delta_{ij}$ as in \cref{sec:Al-case}, and Lagrange interpolation on even polynomials of degree $2d$ in $P$ gives:

%\begin{prop}
\beq
\label{eq:BCD-derivative-formula}
\partial_{u_r}\lambda_D(P)=\frac{2P\,(P^2-1)}{(P^2-q_r^2)}\,\frac{\lambda_D'(P)}{\kappa_r}.
\eeq
%\end{prop}

\begin{prop}\label{prop:BCD-jacobi-matrix}
We have, for $a\neq 0$ and $r=0,\dots, d$,
\begin{equation*}
\partial_{u_r}\xi_a=\frac{2\,\xi_a\,(\xi_a^2-1)}{(q_r^2-\xi_a^2)\,\kappa_r}\,,
\qquad
\frac{\partial_{u_r}a_0}{a_0}=\frac{\delta_{r0}}{\lambda_D(q_0)}
+\frac{4}{\kappa_r}\sum_{j=1}^{d}\frac{m_j\,\xi_j^2\,(\xi_j^2-1)}{(q_0^2-\xi_j^2)(q_r^2-\xi_j^2)}\,.
\end{equation*}
\end{prop}

\begin{proof}
The first expression follows from  logarithmic differentiation of \eqref{eq:BCD-stratum-superpotential} in $u_r$ (the blocks at $P=0$ and $P^2=1$ do not move), comparison with \eqref{eq:BCD-derivative-formula}, division by $(P^2-\xi_a^2)^{m_a-1}$ and evaluation at $P=\xi_a$, using the limit relation \[\frac{\lambda_D}{(P^2-\xi_a^2)^{m_a}}\bigg|_{\xi_a}=\frac{\lambda_D'}{2m_a\xi_a(P^2-\xi_a^2)^{m_a-1}}\bigg|_{\xi_a}\,.\] The second expression arises from  evaluation at the node $P=q_0$ (any node may be labelled $0$; $\lambda_D(q_0)\neq0$ generically on $D$), using $\partial_{u_r}\lambda_D(P)|_{P=q_0} = \delta_{r0}$ and substituting in the already established expression for $\partial_{u_r}\xi_j$. 
%For $a=0$: logarithmic differentiation and evaluation at a critical point, inserting the case $a\neq0$.
\end{proof}

\begin{lem}%[canonical coordinates]
In the canonical chart for $D$, we have
\label{lem:BCD-canonical-coords}
\[\eta_D(\partial_{u_i},\partial_{u_j})=\frac{2\,(-1)^{\bar{k}}\,\delta_{ij}}{\kappa_i},
%\qquad
%\partial_{u_i}\cdot\partial_{u_j}=\delta_{ij}\,\partial_{u_j}.
\]
\end{lem}

\begin{proof}
As in Lemma~\ref{lem:An-canonical-coords}, this follows from using \eqref{eq:BCD-derivative-formula} and the residue formulas \eqref{eq:residue-formulas} with $\phi=P^2-1$, and $\epsilon_{\RR}=(-1)^{\bar{k}}$ supplies the sign. For a node $q_i\notin\{0,\pm1\}$ the residues at $P=\pm q_i$ contribute $1/((q_i^2-1)\lambda_D''(q_i))=1/\kappa_i$ each. For the node at $P=0$ ($m_0=N_0$) the single residue at $P=0$ is $-1/\lambda_D''(0)=2/\kappa_0$. For the node at $P^2=1$ ($n=M$), $\lambda_D'(\pm1)\neq0$ and the integrand has simple poles at $P=\pm1$ from $\phi$, each with residue $1/(2\lambda_D'(1))=1/\kappa_i$. Finally, the last two evaluations use Lemma~\ref{lem:BCD-critical-derivative}.
%; the trilinear computation and Proposition~\ref{prop:star-restriction} give the multiplication rule.
\end{proof}
The following lemma will be useful for obtaining the subsequent determinant of Theorem \ref{thm:BCD-determinant}. 

\begin{lem}\label{lem:BCD-product-identities}
We have
\begin{equation*}
\prod_{i=0}^{d}q_i^2=\frac{m_0-N_0}{\bar{k}}\prod_{a=1}^{d}\xi_a^2,
\qquad
\prod_{i=0}^{d}(q_i^2-1)=\frac{M-n}{\bar{k}}\prod_{a=1}^{d}(\xi_a^2-1),
\end{equation*}
\begin{equation*}
\prod_{i=0}^{d}(\xi_a^2-q_i^2)=\frac{\xi_a^2\,m_a(\xi_a^2-1)}{\bar{k}}\prod_{i\neq a}(\xi_a^2-\xi_i^2).
\end{equation*}
Furthermore,
\[\frac{\prod_{i=0}^{d}\kappa_i}{\prod_{i<j}(q_i^2-q_j^2)^2}
=C\,a_0^{d+1}\prod_{a=1}^{d}\xi_a^{\,2(m_a-1-N_0+m_0)}\,(\xi_a^2-1)^{m_a-1-M+n}
\prod_{a}\prod_{i\neq a}(\xi_a^2-\xi_i^2)^{m_a-1},
\]
with $C$ an explicit non-zero constant.
\end{lem}
\begin{proof}
The first three formulae are obtained by comparing the numerator of the bracket in the proof of Lemma~\ref{lem:BCD-critical-derivative} with $\bar k\prod_i(P^2-q_i^2)$ at $P=0$, $P=1$, and $P=\xi_a$, respectively. The final expression arises by writing $\kappa_i$ as in
Lemma~\ref{lem:BCD-critical-derivative} and using the three identities already established; when $m_0=N_0$ (resp.\ $n=M$) the first (resp.\ second) identity reads $0=0$ and is not needed, the corresponding power of $\prod_i q_i^2$ (resp.\ $\prod_i(q_i^2-1)$) in $\prod_i\kappa_i$ being zero.
\end{proof}

\subsubsection{The determinant}

\begin{thm}\label{thm:BCD-determinant}
%In the coordinates $(a_0,\xi_1,\dots,\xi_d)$ on $D$, 
In the frame of linear coordinates on $D$ we have
%converted to the $x$-frame by a unit Jacobian factor $\alpha$,
\begin{equation}\label{eq:BCD-determinant}
\det\eta_D
=\alpha\,a_0^{\,d+1}\,
\prod_{a=1}^{d}\xi_a^{\,2(m_a+m_0-N_0)}\,\big(\xi_a^2-1\big)^{m_a+n-M}
\prod_{1\le a<b\le d}\big(\xi_a^2-\xi_b^2\big)^{m_a+m_b},
\end{equation}
where $\alpha$ is a unit independent of $a_0$.
\end{thm}

\begin{proof}
As in Theorem~\ref{thm:An-determinant}, $\det\eta_D(\xi)=(\det B)^{-2}\det\eta_D(u)$ with $\det\eta_D(u)=\prod_i 2(-1)^{\bar k}\kappa_i^{-1}$ by Lemma~\ref{lem:BCD-canonical-coords}. The row of $B$ corresponding to $a_0$ (Proposition~\ref{prop:BCD-jacobi-matrix}) is not of Cauchy shape, but subtracting from it the combination $\sum_{j}2m_j\xi_j(q_0^2-\xi_j^2)^{-1}\,(\partial_{u_r}\xi_j)_r$ of the remaining rows leaves the single non-zero entry $a_0/u_0$ in the column $r=0$, whence
\begin{equation*}
\det B=\frac{a_0}{u_0}\,\det\big(\partial_{u_i}\xi_a\big)_{1\le a\le d,\;1\le i\le d},
\end{equation*}
the latter being a Cauchy determinant in the variables $(\xi_a^2,q_i^2)_{i\neq0}$ up to rescaling the rows by $2\xi_a(\xi_a^2-1)$ and the columns by $\kappa_i^{-1}$. By the Cauchy formula,
\begin{equation*}
(\det B)^{-2}\det\eta_D(u)=c\,
\frac{u_0^2\prod_{i\neq0}(q_0^2-q_i^2)^2\prod_{i\neq0}\prod_a(q_i^2-\xi_a^2)^2 a_0^{2}\prod_{i=0}^d\kappa_i}{\kappa_0^2\prod_{i<j}(q_i^2-q_j^2)^2 \prod_a\xi_a^2(\xi_a^2-1)^2\prod_{a<b}(\xi_a^2-\xi_b^2)^2}
\end{equation*}
for a non-zero constant $c$. Using the definitions of $u_0=\lambda_D(q_0)$ and $\kappa_0$ and assembling with Lemma \ref{lem:BCD-product-identities} and the sign identity used in Theorem~\ref{thm:An-determinant} yields \eqref{eq:BCD-determinant} up to the factor $a_0^{-2}\prod_a(\xi_a^2-1)^{-1}$, which is exactly supplied by the passage from the coordinates $(a_0,\xi)$ to the $x$-frame: $x_{l+1}\mapsto a_0$ has Jacobian proportional to $a_0$, and the change of variables $\xi_j=\cosh(v_j/2)$ (up to sign) has triangular Jacobian with diagonal entries proportional to $(\xi_j^2-1)^{1/2}$.
\end{proof}

We see that $\xi_a^2-\xi_b^2\equiv\mathsf{e}_{v_a-v_b}\mathsf{e}_{v_a+v_b}$ carries the
mirrors $\{v_a=\pm v_b\}$ with total exponent $m_a+m_b$, exactly as at a generic wall in the finite
Coxeter-group setting. As far as the remaining factors are concerned, since $p_j=c_{v_j}$, we have the following equivalences:
%$\mathsf{e}_{2\gamma}=2\,\mathsf{e}_\gamma c_\gamma$,
\[
\xi_a^2-1\equiv\mathsf{e}_{v_a}^{2},\qquad
\xi_a\equiv c_{v_a},\qquad
\xi_a^2\big(\xi_a^2-1\big)\equiv\mathsf{e}_{2v_a}^{2}\,.
\]
% $c_{v_a}$, vanishing on the component $\{\re^{v_a}=-1\}$ of $\{\mathsf{e}_{2v_a}=0\}$, is not a
%unit. 
Write $k_a:=m_a+m_0-N_0$ and $j_a:=m_a+n-M$ for the two exponents in
\eqref{eq:BCD-determinant}, the pole-free value $m_a$ corrected by the order of $\lambda_D$ along
$P=0$ and along $P^2=1$; for the canonical markings $M,N_0\le1$, so both are non-negative, and
$j_a-k_a=N_0-M-m_0+n$. Pairing $\min(j_a,k_a)$ copies of $\xi_a^2$ with as many of $\xi_a^2-1$,
\[
\xi_a^{2k_a}\big(\xi_a^2-1\big)^{j_a}\;\equiv\;
\mathsf{e}_{2v_a}^{\,2\min(j_a,k_a)}\cdot
\begin{cases}
\mathsf{e}_{v_a}^{\,2(j_a-k_a)}, & j_a\ge k_a,\\[2pt]
c_{v_a}^{\,2(k_a-j_a)}, & k_a\ge j_a,
\end{cases}
\]
which for the linear strata ($m_0=0$) gives $\mathsf{e}_{2v_a}^{2(m_a-1)}\mathsf{e}_{v_a}^{2(n+1)}$, $\mathsf{e}_{2v_a}^{2m_a}\mathsf{e}_{v_a}^{2n}$
and $\mathsf{e}_{2v_a}^{2(m_a-1)}\mathsf{e}_{v_a}^{2n}$ in types $B_l$, $C_l$, $D_l$ respectively, and $\mathsf{e}_{2v_a}^{2m_a}c_{v_a}^{2m_0}$ for the translates of the long-root strata in type $C_l$ ($n=0$) mentioned above. Each factor is of the form 
$\mathsf{e}_\beta|_D$ for a transverse root $\beta$, and all identifications below hold up to a sign
and hence up to a unit. Writing $i,i'$ for indices among the $m_a$ with $p_\bullet^2=\xi_a^2$:
$\mathsf{e}_{v_a}=\mathsf{e}_{\epsilon_i}$ occurs only in type $B$, where $\epsilon_i$ is a short root, and the $2n$ roots $\epsilon_i\pm\epsilon_j$ with $p_j^2=1$, i.e. $v_j=0$, each restrict to $v_a$, accounting for $\mathsf{e}_{v_a}^{2n}$;
$\mathsf{e}_{2v_a}$ is the long root $2\epsilon_i$ in type $C$, while in types $B,D$ its exponent is
non-zero only for $m_a\ge2$, and then whichever of $\epsilon_i\pm\epsilon_{i'}$, $i\neq i'$, avoids
$\RR_D$ restricts to $2v_a$; and the $2m_0$ roots $\epsilon_i\pm\epsilon_j$ with $p_j=0$ each restrict
to $c_{v_a}$, since $\re^{v_j}=-1$ there, accounting for $c_{v_a}^{2m_0}$. As
\eqref{eq:main-factorization} constrains only the total exponent at each restricted root, exhibiting
one such $\beta$ per factor suffices. The prefactors $a_0^{\,d+1}$ and $\alpha$ are units. We have
deduced the following 
\begin{cor}\label{cor:BCD-main-thm}
Theorem~\ref{thm:main} holds for $\RR \in \{B_l,C_l,D_l\}$.
\end{cor}

%\begin{rmk}\label{rem:residue-general}
%The residue method applies to any meromorphic superpotential; in particular an analogous derivation is expected to hold for the generalised extensions of \cite{Zuo2020,MaZuo2022}.
%\end{rmk}

\section{Exceptional series}

%This part establishes the factorisation of $\det\eta_D$ described in Theorem~\ref{thm:main} for strata of codimension one, two, and three (the last under the simply laced hypothesis), uniformly in the Dynkin type, together with explicit formulas for the determinant in terms of Jacobian minors. The argument follows the strategy introduced in \cite[§6--7]{AntoniouFS:2020} for finite Coxeter groups, developed here from first principles in the trigonometric setting. The letters $\lambda,\nu,\theta$ denote simple roots in Section 10. We do not attempt to identify the resulting exponents $n_{\beta,D}$ with Coxeter numbers of root subsystems: such an identification, while true for several of the classical strata treated in Part I and for the mirrors within $\mathcal{R}_D$ itself, does not hold for a general transverse root $\beta$, as counterexamples show.

\subsection{Preliminaries}

%\subsubsection{Jacobian minors}
\subsubsection{$\RR=G_2$}
Our study in the previous section immediately implies that \cref{thm:main} holds for $\RR=G_2$. Recall that the LG superpotential for $\RR=G_2$ is obtained from the case $\RR=D_4$ upon restriction to the triality invariant locus $\tilde x_1=\tilde x_3=\tilde x_4$. As a bonus, the arguments of \cref{thm:BCD-determinant,cor:BCD-main-thm}, applied verbatim to this case, yield the following
\begin{cor}
\cref{thm:main} holds for $\RR=G_2$.
\end{cor}
\subsubsection{Jacobian minors and the contravariant Saito metric}

Throughout this section we differentiate in the coordinates $(\tilde x_1,\dots,\tilde x_l;x_{l+1})$ of \eqref{eq:tildex-coords}, for which $\partial/\partial\tilde x_i=\partial_{\alpha_i}$, and we abbreviate $\tilde x_{l+1}:=x_{l+1}$. Moreover,  Fourier monomials continue to be written as $\re^{\mu(x)}$, that is, with exponents in the frame $(x_1,\dots,x_l)$ of \cref{sec:killing-normalization}, so that $\mu=\sum_i n_i\alpha_i$ contributes $\re^{\sum_i n_ix_i}$. The two frames are related by the constant matrix $\partial x_i/\partial\tilde x_j=(\alpha_i,\alpha_j)$, so the Gram determinant of $\eta$ computed in the $\tilde x$-frame and in the linear chart $(x_1,\dots,x_l;x_{l+1})$ of \cref{sec:intro-strata} differ by the non-zero constant $\det\big((\omega^\vee_i,\omega^\vee_j)\big)^{2}$, which is immaterial for \eqref{eq:main-factorization}.

Since $t_{l+1}=x_{l+1}$ by \eqref{eq:killing-normalization}, the row of the Jacobian matrix $T=(\partial t_i/\partial\tilde x_j)_{i,j\le l+1}$ corresponding to $t_{l+1}$ equals $(0,\dots,0,1)$. The cofactor expansion along this row gives $\det T=\pm \,J$, where
\begin{equation*}
J:=\det\big(\partial_{\alpha_j}t_i\big)_{1\le i,j\le l}
\end{equation*}
is the determinant of the finite block. In the following we order the flat coordinates so that $t_{\bar k} = t_1$ occupies the first position among $t_1,\dots,t_l$ to be compatible with the anti-diagonal form \eqref{eq:killing-normalization} of the Gram matrix of the Saito metric, and let $J_m$ denote the minor of the finite block obtained by deleting the first row and $m$-th column.

\begin{prop}\label{prop:cramer-representation}
The identity field has no component in the extended direction, and
\begin{equation}\label{eq:cramer-representation}
e \;=\; \sum_{i=1}^{l}e^{i}\,\partial_{\alpha_i},
\qquad
e^{i}=-(-1)^{i}\,J^{-1}J_i .
\end{equation}
\end{prop}

\begin{proof}
Define $X:=\sum_i(-1)^{1+i}J^{-1}J_i\,\partial_{\alpha_i}$. For $j\le l$, $j\neq\bar k$, the quantity $\sum_i(-1)^{1+i}J_i\,\partial_{\alpha_i}t_j$ is the Laplace expansion, along its first row, of the determinant of the matrix obtained from the finite block by replacing the row of $t_{\bar k}=t_1$ with a second copy of the row of $t_j$; having two equal rows, this determinant vanishes, so $X(t_j)=0$. For $j=\bar k$ the same expansion reproduces $J$ itself, so $X(t_{\bar k})=1$. Finally $X(t_{l+1})=0$ since $\partial_{\alpha_i}t_{l+1}=0$. As $X$ contains no $\partial_{x_{l+1}}$, and the vector field $\partial_{t_{1}}$ likewise has vanishing $x_{l+1}$-component, the equality $X=\partial_{t_{1}}=e$ follows from agreement on the coordinate system $(t_1,\dots,t_{l+1})$.
\end{proof}

%\subsection{The contravariant Saito metric}

\begin{prop}\label{prop:contravariant-metric}
In the coordinates $\tilde x_i$, the finite components of the contravariant Saito metric are
\begin{equation}\label{eq:contravariant-metric}
\eta^{ij}
=-\,\partial_{\omega^\vee_j}e^i-\partial_{\omega^\vee_i}e^j
=(-1)^{j}\,\partial_{\omega^\vee_i}\frac{J_j}{J}
+(-1)^{i}\,\partial_{\omega^\vee_j}\frac{J_i}{J},
\qquad 1\le i,j\le l.
\end{equation}
\end{prop}

\begin{proof}
For the Lie derivative of the contravariant tensor $\gamma$ in the coordinates $\tilde x$,
\begin{equation*}
(\mathcal{L}_e\gamma)^{ij}=e\big(\gamma^{ij}\big)-\gamma^{aj}\,\partial_{\tilde x_a}e^i-\gamma^{ia}\,\partial_{\tilde x_a}e^j .
\end{equation*}
The components of $e$ in the $\tilde x$-system are $e^i$ ($i\le l$) and $0$ ($i=l+1$) by Proposition~\ref{prop:cramer-representation} and \eqref{eq:tildex-coords}. The tensor $\gamma$ is constant in $\tilde x$ and block-diagonal, so the first term vanishes and the sums over $a$ run over $a\le l$ for $i,j\le l$. In the frame $\{\partial_{\alpha_a}\}$ the covariant components of $\gamma$ are $\gamma_{ab}=(\alpha_a,\alpha_b)$, hence the contravariant components in the coordinates $\tilde x$ are $\gamma^{aj}=(\omega^\vee_a,\omega^\vee_j)$, and by \eqref{eq:coweight-derivative}, $\gamma^{aj}\partial_{\tilde x_a}=\partial_{\omega^\vee_j}$. The second equality in \eqref{eq:contravariant-metric} is \eqref{eq:cramer-representation}.
\end{proof}

%\subsection{The restriction of $\eta$ to a stratum}

\begin{lem}\label{lem:stratum-coords}
For $I\subset\{1,\dots,l\}$ and $D=\bigcap_{q\in I}\Pi_{\alpha_q}$, the functions $\tilde x_i$ ($i\notin I$) together with $x_{l+1}$ restrict to a coordinate system on $D\times\mathbb{C}$.
\end{lem}

\begin{proof}
One has $D=\langle\omega^\vee_i:i\notin I\rangle$, since $(\omega^\vee_i,\alpha_q)=0$ for $i\notin I$, $q\in I$, and the dimensions agree. A linear relation $\sum_{i\notin I}a_i\tilde x_i=0$ on $D$, evaluated at the points $x=\omega^\vee_j$ ($j\notin I$), reads $\sum_i a_i(\omega^\vee_i,\omega^\vee_j)=0$ for all $j$; the Gram matrix of the coweights $\{\omega^\vee_i\}_{i\notin I}$ is nondegenerate (they are linearly independent and the form is nondegenerate on the real span), so $a_i=0$.
\end{proof}

\begin{prop}
%[Jacobi minor identity]
\label{prop:jacobi-minor-identity}
Let $M$ be an invertible symmetric $(l+1)\times(l+1)$ matrix, $I\subset\{1,\dots,l+1\}$, $T$ its complement. Then
\begin{equation*}
\det\big((M^{-1})_{TT}\big)=\frac{\det(M_{II})}{\det M}.
\end{equation*}
\end{prop}

\begin{proof}
The result follows from the Schur complement factorisation $\det M=\det(M_{II})\det\big(M_{TT}-M_{TI}M_{II}^{-1}M_{IT}\big)$, and the block-inverse formula $(M^{-1})_{TT}=\big(M_{TT}-M_{TI}M_{II}^{-1}M_{IT}\big)^{-1}$.
\end{proof}

\begin{thm}
%[Determinant formula]
\label{thm:determinant-formula}
Let $I=\{i_1,\dots,i_\kappa\}\subset\{1,\dots,l\}$ and $D=\bigcap_{q\in I}\Pi_{\alpha_q}$. Set
\begin{equation}\label{eq:P-definition}
P \;=\; J^{2}\,\det\big(\eta^{i_ai_b}\big)_{1\le a,b\le\kappa},
\end{equation}
with $\eta^{ij}$ as in \eqref{eq:contravariant-metric}. Then $P$ has a well-defined restriction $P_D=P|_D$, and the determinant of $\eta_D$ in the coordinates of Lemma~\ref{lem:stratum-coords} satisfies
\begin{equation}\label{eq:detD-formula}
\det\eta_D=\varepsilon\, P_D,
\end{equation}
where $\varepsilon$ is a fixed unit depending only on the chosen normalisations.
\end{thm}

\begin{proof}
The Gram matrix of the Saito metric \eqref{eq:killing-normalization} in the $\tilde x$-coordinates is
%$\eta=\sum_{i}dt^i\,dt^{\,l+2-i}$ gives 
\[\eta_{rs}=\sum_i \partial_{\tilde x_r}t^i\,\partial_{\tilde x_s}t^{\,l+2-i}\,,\]
where the coefficients are Fourier polynomials. For $r=i_a\in I$ and any invariant $t^i$, Proposition~\ref{prop:divisibility} below gives $\partial_{\tilde x_r}t^i=\partial_{\alpha_{i_a}}t^i$ divisible by $\mathsf{e}_{\alpha_{i_a}}$, hence vanishing on $D$. Consequently, the pull-back of each differential $dt^i$ to $D$ involves only the tangential differentials $d\tilde x_r$ ($r\notin I$) and $dx_{l+1}$, and the matrix of $\eta_D$ in the coordinates of Lemma~\ref{lem:stratum-coords} is the tangential submatrix $C=(\eta_{rs})_{r,s\in T}\big|_D$, $T=(\{1,\dots,l\}\setminus I)\cup\{l+1\}$.

Let $Q=(\eta^{rs})_{r,s\le l+1}$ be the full contravariant matrix in the $\tilde x$-coordinates, so that $(\eta_{rs})=Q^{-1}$. By Proposition~\ref{prop:jacobi-minor-identity},
\begin{equation*}
\det C=\frac{\det\big((\eta^{i_ai_b})_{a,b}\big)}{\det Q}.
\end{equation*}
Now, \[\det Q^{-1}=\det(\eta_{rs})=\big(\det(\partial t/\partial\tilde x)\big)^{2}\det(\eta_{\alpha\beta})=\pm J^2\,,\] using \eqref{eq:killing-normalization} and $\det T=\pm J$. Hence \[\det C=\pm J^{2}\det\big((\eta^{i_ai_b})\big)=\pm P\,;\] the left-hand side is the determinant of a matrix of polynomials restricted to $D$, so the restriction $P_D$ exists, and \eqref{eq:detD-formula} holds with $\varepsilon=\pm1$. The entries $\eta^{i_ai_b}$ appearing in $P$ are the finite ones of \eqref{eq:contravariant-metric}: indices in $I$ are finite indices, and the extended row and column of $Q$ do not enter the minor.
\end{proof}

The individual entries of the minor in \eqref{eq:P-definition} may be singular on $D$. The restriction of $P$ is computed below by restricting one transverse direction at a time in a fixed order, with the existence of the total limit being guaranteed by Theorem~\ref{thm:determinant-formula}.

%\section{Factorisation, divisibility, and combinatorial preliminaries}

%\subsection{Factorisation of the Jacobian}
Recall that there
%\begin{prop}\label{prop:jacobian-factorization}
%There 
is a proportionality of Fourier polynomials \cite{B02}
\begin{equation}\label{eq:jacobian-factorization}
J \propto \re^{(d_1+\dots +d_l)x_{l+1}}\prod_{\beta\in\mathcal{R}_+}\mathsf{e}_\beta = \re^{(l+1) d_{\bar k} x_{l+1}/2} \prod_{\beta\in\mathcal{R}_+}\mathsf{e}_\beta.
\end{equation}
%\end{prop}

%\begin{proof}
%By \cite{DZ1} the invariants $t_1,\dots,t_l$ are obtained from the fundamental characters $\chi_1,\dots,\chi_l$ of $\mathcal{R}$ by an invertible transformation which is triangular over the invariant ring with unit coefficients (powers of $w_0$); such a transformation multiplies the Jacobian by a unit. For the fundamental characters, the classical computation of Steinberg identifies $\det(\partial_{\alpha_j}\chi_i)$ with the Weyl denominator: each $\chi_i=\sum_{w\in\mathcal{W}}\re^{w\omega_i(x)}+\dots$ is $\mathcal{W}$-invariant, so the determinant changes by $\det(w)=\pm1$ under the Weyl action on $x$ and is therefore skew-invariant; every skew-invariant Fourier polynomial vanishes on each mirror $\{\re^{\beta(x)}=1\}$ and, by primitivity of $\beta$ in the root lattice, is divisible by $\mathsf{e}_\beta$; distinct mirrors give coprime factors, so the Weyl denominator $\prod_{\beta\in\mathcal{R}_+}\mathsf{e}_\beta$ (itself skew) divides the determinant; finally, the extreme exponents of both sides equal $\rho=\frac12\sum_{\beta>0}\beta$ after symmetrisation, so the quotient is a unit.
%\end{proof}

%\subsection{Divisibility and vanishing of the minors}

We will repeatedly need to identify the ideal cut out by a trigonometric form. We record the relevant statement once, in the generality in which it will be used: the exponent is allowed to be non-primitive, which is what occurs upon restriction to a stratum.

\begin{lem}\label{lem:cyclotomic-ideal}
Let $\Lambda$ be a lattice of characters of a complex torus, and let $0\neq\delta\in\Lambda$. Then the ideal of the reduced hypersurface $\{\re^{\delta}=1\}$ in the ring of Fourier polynomials with exponents in $\Lambda$ is generated by $\re^{\delta}-1$, equivalently by $\mathsf{e}_\delta$, whether or not $\delta$ is primitive in $\Lambda$.
\end{lem}

\begin{proof}
Write $\delta=r\gamma$ with $\gamma\in\Lambda$ primitive and $r\in\bbZ_{>0}$. Being primitive, $\gamma$ can be completed to a $\bbZ$-basis of $\Lambda$, so that $\re^{\gamma}$ is one of a system of coordinates on the torus; the loci $\{\re^{\gamma}=\zeta\}$, $\zeta^r=1$, are therefore smooth irreducible hypersurfaces with principal, pairwise coprime ideals $(\re^{\gamma}-\zeta)$. Since $\{\re^{\delta}=1\}$ is their disjoint union, its ideal is generated by
\[
\prod_{\zeta^r=1}\big(\re^{\gamma}-\zeta\big)=\re^{r\gamma}-1=\re^{\delta}-1\,,
\]
which differs from $\mathsf{e}_\delta$ by a unit.
\end{proof}

\begin{prop}\label{prop:divisibility}
Let $p$ be a $\widetilde{\mathcal{W}}^{(\bar k)}$-invariant Fourier polynomial and $\alpha\in\mathcal{R}$. Then $\mathsf{e}_\alpha$ divides $\partial_\alpha p$.
\end{prop}

\begin{proof}
Invariance under the reflection $s_\alpha$ gives $p(x)=p(s_\alpha x)$. Applying $\partial_\alpha$ and using $s_\alpha\alpha=-\alpha$ yields $(\partial_\alpha p)(x)=-(\partial_\alpha p)(s_\alpha x)$, so $\partial_\alpha p$ vanishes on the fixed locus $\{\alpha(x)=0\}$. The same argument applies verbatim to every affine reflection of $\widetilde{\mathcal{W}}^{(\bar k)}$ with linear part $s_\alpha$: these are the reflections in the affine hyperplanes $\{\alpha(x)\in2\pi\mathrm{i}\,\bbZ\}$, and their union is precisely the subvariety $\{\re^{\alpha(x)}=1\}$ of the torus. Hence $\partial_\alpha p$ vanishes on the whole of $\{\re^{\alpha(x)}=1\}$, and \cref{lem:cyclotomic-ideal} applied to the weight lattice and $\delta=\alpha$ gives the claim.
\end{proof}

\begin{prop}\label{prop:minor-divisibility}
Let $U=\langle\alpha_i:i\neq m\rangle$. Then $\mathsf{e}_\alpha$ divides $J_m$ for every $\alpha\in\mathcal{R}\cap U$. Moreover $J_m$ vanishes on every stratum $\Pi_\beta\cap\Pi_\gamma$ with $\beta,\gamma\in\mathcal{R}$, $\beta\neq\pm\gamma$.
\end{prop}

\begin{proof}
Write $\alpha=\sum_{i\neq m}c_i\alpha_i$. The linear combination of the columns of the matrix defining $J_m$ with coefficients $c_i$ has entries $\partial_\alpha t_j$, all divisible by $\mathsf{e}_\alpha$ by Proposition~\ref{prop:divisibility}; expanding the determinant along this combination gives the first claim. For the second, some nontrivial combination $a_1\beta+a_2\gamma$ lies in $U$; since $\partial_\beta t_j$ and $\partial_\gamma t_j$ both vanish on $\Pi_\beta\cap\Pi_\gamma$, so does the corresponding column combination, hence so does $J_m$.
\end{proof}

%\subsection{Restriction of the minors}

\begin{prop}\label{prop:minor-restriction}
For every $\beta\in\mathcal{R}_+\setminus\{\alpha_m\}$, $\mathsf{e}_{\beta|_{D_m}}$ divides $J_m\big|_{D_m}$.
\end{prop}

\begin{proof}
By Proposition~\ref{prop:divisibility}, $\mathsf{e}_{\alpha_m}$ divides $\partial_{\alpha_m}t_i$ and $\mathsf{e}_{\beta}$ divides $\partial_{\beta}t_i$; in particular, both columns vanish on every component of $\{\re^{\alpha_m(x)}=1\}$ and $\{\re^{\beta(x)}=1\}$, respectively. The column-combination argument in the proof of Proposition~\ref{prop:minor-divisibility}, applied to the pair $(\alpha_m,\beta)$, then shows that $J_m$ vanishes at every point of the subvariety
\[
\big\{\re^{\alpha_m(x)}=1\big\}\cap\big\{\re^{\beta(x)}=1\big\}
\]
of the torus, whose shadow on $D_m$ is the full reduced hypersurface $\{\re^{\beta|_{D_m}}=1\}$. By \cref{lem:cyclotomic-ideal}, applied to the restricted weight lattice of $D_m$ and to $\delta=\beta|_{D_m}$, the ideal of that hypersurface is generated by $\mathsf{e}_{\beta|_{D_m}}$.
%; note that this does not require $\beta|_{D_m}$ to be primitive, which in general it is not.
\end{proof}

%\begin{proof}
%By Proposition~\ref{prop:minor-divisibility}, $J_k$ vanishes on $\Pi_{\alpha_k}\cap\Pi_\beta$, which is the zero locus in $D_k$ of the restricted form $\beta|_{D_k}$. Since $\beta|_{D_k}$ is primitive in the restricted weight lattice, $\mathsf{e}_{\beta|_{D_k}}$ generates the corresponding ideal in the Fourier ring of $D_k$, giving the claim.
%\end{proof}

\begin{lem}\label{lem:no-proportional-roots}
Let $\mathcal{R}$ be simply laced, and let $\alpha,\beta,\beta'\in\mathcal{R}$ with $\beta,\beta'\notin\{\pm\alpha\}$. If $\beta-c\,\beta'\in\mathbb{R}\alpha$ for some $c\in\mathbb{R}$, then $c\in\{1,-1\}$.
\end{lem}

\begin{proof}
Let $W=\mathrm{span}(\alpha,\beta')$, a two-dimensional subspace, so that the hypothesis says exactly $\beta\in\mathcal{R}\cap W$. Since $\mathcal{R}$ is simply laced, $(\alpha,\beta')\in\{0,\pm1\}$ (as $\beta'\ne\pm\alpha$), and $\mathcal{R}\cap W$ is a rank-two simply laced root system, hence of type $A_1\times A_1$ or $A_2$; in either case it is completely determined by $\alpha,\beta'$. If $(\alpha,\beta')=0$, $\mathcal{R}\cap W=\{\pm\alpha,\pm\beta'\}$ exactly (type $A_1\times A_1$). If $(\alpha,\beta')=\mp1$, set $\gamma=\alpha\pm\beta'$; then $(\gamma,\gamma)=(\alpha,\alpha)\pm2(\alpha,\beta')+(\beta',\beta')=2\mp2+2=2$, so $\gamma$ is itself a root, and $\mathcal{R}\cap W=\{\pm\alpha,\pm\beta',\pm\gamma\}$ exactly (type $A_2$), since a rank-two simply laced root system containing $\alpha,\beta'$ with this inner product has exactly six roots. In both cases, since $\beta\in\mathcal{R}\cap W$ and $\beta\ne\pm\alpha$, either $\beta=\pm\beta'$, giving $\beta-(\pm1)\beta'=0=0\cdot\alpha$ and hence $c=\pm1$; or (in the second case) $\beta=\pm\gamma=\pm(\alpha\pm\beta')$, giving $\beta-(\pm1)\beta'=\pm\alpha$ (same signs throughout) and hence again $c=\pm1$. No other roots lie in $W$, so these are the only possibilities.
\end{proof}

Let $\cA_m=\cA \cap D_m$ be the restriction of the root hyperplane arrangement $\cA$ to $D_m$. Note that non-proportional roots $\beta$, $\beta' \in \RR$ may restrict to proportional roots on $D_m$, and the proportionality factor need not have absolute value 1. For $H \in \cA_m$, we shall write $\beta_H \in D_m^*$ for the {\it longest} restricted positive root $\beta|_{D_m} \in D_m^*$ with kernel $H$.

\begin{prop}
    Let $\RR\in \{E_6,E_7,E_8, F_4\}$. Then, for $m \in \llbracket 1,l\rrbracket$,

    \beq
    J_m\Big|_{D_m} \propto \re^{(l-1) d_{\bar k} x_{l+1}/2} \prod_{H \in \cA_m} \mathsf{e}_{\beta_H}\,.
    \label{eq:Jkequiv}
    \eeq
    \label{prop:Jkequiv}
\end{prop}

\begin{proof}
By \cref{prop:minor-restriction}, we already know that each individual factor on the right-hand side of \eqref{eq:Jkequiv} divides $J_m$. The hyperplanes in the product are distinct, and we select exactly one restricted root per hyperplane: writing $\beta_H=r_H\gamma_H$ with $\gamma_H$ primitive, the irreducible factors $\re^{\gamma_H}-\zeta$, $\zeta^{r_H}=1$, of distinct $\mathsf{e}_{\beta_H}$ involve non-proportional $\gamma_H$ and are therefore pairwise coprime. Hence $J_m$ is divisible by the entire product. 
At the same time, a direct inspection of each case reveals that 
\[
\deg_{\re^{\pm x_i}} J_m = \deg_{\re^{\pm x_i}} \prod_{H \in \cA_m} \mathsf{e}_{\beta_H}\,, \quad i \in \llbracket 1, l\rrbracket \setminus \{m\}
%\label{eq:degmatch}
\]
as Fourier polynomials in $(x_1, \dots, x_l)$, which in particular have identical Newton polytopes. We conclude that $J_m$ is a non-zero constant multiple of the right-hand side of \eqref{eq:Jkequiv} up to a monomial in $\re^{x_{l+1}}$, which is easily computed as $\re^{(l-1) d_{\bar k}/2 x_{l+1}}$ from the definition of $J_m$.
\end{proof}

The right-hand side of \eqref{eq:Jkequiv} is a product of \emph{non-zero} linear forms on $D_m$, one for each hyperplane of $\cA_m$, and is in particular not identically zero there. We record the consequence separately: it is what guarantees that the component $e^m$ of the identity field is genuinely singular along $\Pi_{\alpha_m}$, and hence that the reduced minors of \cref{defn:reduced-minors} below do not acquire a further factor of $\mathsf{e}_{\alpha_m}$.

\begin{cor}\label{prop:minor-nonvanishing}
Let $\RR$ be as in \cref{prop:Jkequiv}. Then $J_m$ does not vanish identically on $\Pi_{\alpha_m}$. Equivalently, $\mathsf{e}_{\alpha_m}\nmid J_m$ for every $m$.
\end{cor}

\begin{proof}
Immediate from \eqref{eq:Jkequiv}: every $\beta_H$, $H\in\cA_m$, restricts to a non-zero linear form on $D_m$, so \[\prod_{H\in\cA_m}\mathsf{e}_{\beta_H}\not\equiv0\] on $D_m$, whence $J_m|_{D_m}\not\equiv0$.
\end{proof}

\begin{defn}\label{defn:reduced-minors}
We define the reduced Jacobian and reduced minors as
\begin{equation*}
I := J\prod_{\alpha\in\Delta}\mathsf{e}_\alpha^{-1},
\qquad
I_m := J_m\prod_{\alpha\in\Delta\setminus\{\alpha_m\}}\mathsf{e}_\alpha^{-1}.
\end{equation*}
\end{defn}
These are Fourier polynomials by Proposition \ref{prop:minor-divisibility}, with $I\equiv\prod_{\beta\in\mathcal{R}_+\setminus\Delta}\mathsf{e}_\beta$. By \eqref{eq:trig-derivative}, $\partial_{\omega^\vee_i}\mathsf{e}_{\alpha_m}=\delta_{im}\,c_{\alpha_m}$, and therefore
\begin{equation}\label{eq:reduced-minor-derivative}
\partial_{\omega^\vee_i}\frac{J_m}{J}
=\partial_{\omega^\vee_i}\frac{I_m}{\mathsf{e}_{\alpha_m}I}
=-\,\delta_{im}\,\frac{c_{\alpha_m}}{\mathsf{e}_{\alpha_m}^{2}}\,\frac{I_m}{I}
+\frac{1}{\mathsf{e}_{\alpha_m}}\,\partial_{\omega^\vee_i}\frac{I_m}{I}.
\end{equation}
The factor $c_{\alpha_m}$, absent in the polynomial theory, is an even function of $\alpha_m(x)$, equal to $1$ on $\{\alpha_m(x)=0\}$.

%\subsection{The adjacent-minor identity}

\begin{prop}\label{prop:adjacent-minor}
Let $\alpha_r,\alpha_s\in\Delta$ with $(\alpha_r,\alpha_s)\neq0$, and $D=D_{r,s}=\Pi_{\alpha_r}\cap\Pi_{\alpha_s}$. Then,
\begin{equation}\label{eq:adjacent-minor-identity}
I_s\big|_{D}=(-1)^{\,r-s-1}\,I_r\big|_{D}.
\end{equation}
\end{prop}

\begin{proof}
By Proposition~\ref{prop:divisibility}, write \[\partial_{\alpha_m}t_i=\mathsf{e}_{\alpha_m}Q_{mi}\] with $Q_{mi}$ Fourier polynomials, and let \[v_m=(\partial_{\alpha_m}t_2,\dots,\partial_{\alpha_m}t_{l})^{\mathsf T}\,,\quad Q_m=(Q_{m2},\dots,Q_{m,l})^{\mathsf T}\] denote the columns entering the minors, the row of $t_{\bar k}=t_1$ having been deleted. Computing the mixed second derivative in both orders,
\begin{equation*}
\partial_{\alpha_s}\partial_{\alpha_r}t_i
=(\alpha_r,\alpha_s)\,c_{\alpha_r}Q_{ri}+\mathsf{e}_{\alpha_r}\,\partial_{\alpha_s}Q_{ri}
=(\alpha_r,\alpha_s)\,c_{\alpha_s}Q_{si}+\mathsf{e}_{\alpha_s}\,\partial_{\alpha_r}Q_{si},
\end{equation*}
and on $D$ the $\mathsf{e}$-terms vanish and the $c$-factors equal $1$, so $(\alpha_r,\alpha_s)\neq0$ gives $Q_{ri}|_D=Q_{si}|_D$. Assume now $r<s$. Then $\mathsf{e}_{\alpha_r}^{-1}J_s=\det A_{rs}$, where $A_{rs}$ has columns \[v_1,\dots,v_{r-1},Q_r,v_{r+1},\dots,\widehat{v_s},\dots,v_l\] and $\mathsf{e}_{\alpha_s}^{-1}J_r=\det A_{sr}$, where $A_{sr}$ has columns $v_1,\dots,\widehat{v_r},\dots,v_{s-1},Q_s,v_{s+1},\dots,v_l$. On $D$ the two matrices have the same columns up to the cyclic permutation moving position $r$ to position $s$, of sign $(-1)^{s-r-1}$. Multiplying both minors by $\prod_{\alpha\in\Delta\setminus\{\alpha_r,\alpha_s\}}\mathsf{e}_\alpha^{-1}$ yields \eqref{eq:adjacent-minor-identity}.
\end{proof}

\subsection{Codimension one}

Fix $m\in\{1,\dots,l\}$ and let $D=D_m=\Pi_{\alpha_m}$, with coordinates $\tilde x_i$ ($i\neq m$) and $x_{l+1}$.

\begin{thm}\label{thm:codim1-factorization}
The determinant of the restricted Saito metric satisfies
\begin{equation}\label{eq:codim1-factorization}
\det\eta_D \;\equiv\; J_m\big|_D\;\prod_{\beta\in\mathcal{R}_+\setminus\{\alpha_m\}}\mathsf{e}_{\beta|_D}.
\end{equation}
\end{thm}

\begin{proof}
By Theorem~\ref{thm:determinant-formula} with $I=\{m\}$ and Proposition~\ref{prop:contravariant-metric},
\begin{equation*}
\det\eta_D=\varepsilon\,J^{2}\eta^{mm}\big|_D
=2\varepsilon(-1)^{m}\,J^{2}\,\partial_{\omega^\vee_m}\frac{J_m}{J}\Big|_D
=2\varepsilon(-1)^{m}\big(J\,\partial_{\omega^\vee_m}J_m-J_m\,\partial_{\omega^\vee_m}J\big)\Big|_D .
\end{equation*}
As $J=u\,\mathsf{e}_{\alpha_m}\prod_{\beta\in\mathcal{R}_+\setminus\{\alpha_m\}}\mathsf{e}_\beta$ for a unit $u$, we have $J|_D=0$, which removes the first term. In the second term, expanding by the Leibniz rule, every contribution retains a factor $\mathsf{e}_{\alpha_m}$ except the one in which the derivative falls on $\mathsf{e}_{\alpha_m}$, and $\partial_{\omega^\vee_m}\mathsf{e}_{\alpha_m}=c_{\alpha_m}\to1$ on $D$ by \eqref{eq:trig-derivative}. Hence,
%using $\mathsf{e}_\beta|_D=\mathsf{e}_{\beta|_D}$,
\begin{equation*}
\partial_{\omega^\vee_m}J\big|_D= u\prod_{\beta\in\mathcal{R}_+\setminus\{\alpha_m\}}\mathsf{e}_{\beta|_D} .
\end{equation*}
Substituting gives \[\det\eta_D\equiv \,J_m|_D\prod_{\beta\in\mathcal{R}_+\setminus\{\alpha_m\}}\mathsf{e}_{\beta|_D},\] which is \eqref{eq:codim1-factorization}.
\end{proof}

\begin{cor}
    \cref{thm:main} holds for $\RR\in \{E_6,E_7,E_8,F_4\}$ and $\mathrm{codim}(D)=1$.
\end{cor}

\begin{proof}
    Substituting \eqref{eq:Jkequiv} into \eqref{eq:codim1-factorization} expresses $\det\eta_D$, up to a unit, as a product of trigonometric forms with non-negative integer exponents. Since each $\beta_H$ is itself the restriction to $D$ of a positive root, every factor is of the form $\mathsf{e}_{\beta|_D}$ for some $\beta\in\RR_+\setminus\{\alpha_m\}$, and the product is of the form \eqref{eq:main-factorization}.
\end{proof}
\subsection{Codimension two}
 \label{sec:codim2}
Let $\alpha_r,\alpha_s\in\Delta$ be distinct, $D=D_{r,s}$, $S=\langle\alpha_r,\alpha_s\rangle$, $\mathcal{R}_S=\mathcal{R}\cap S$, and
\begin{equation*}
d_S\coloneqq |\mathcal{R}_+\cap S|-2 .
\end{equation*}
By Theorem~\ref{thm:determinant-formula}, $\det\eta_D=\varepsilon\,J^2\big(\eta^{ss}\eta^{rr}-(\eta^{sr})^2\big)|_D$, the restriction being performed on $\{\alpha_s=0\}$ first and then letting $\alpha_r\to0$.
 
\begin{lem}\label{lem:rank2-positive-roots}
 We have $I\equiv \tilde f\,g$, with
\begin{equation*}
\tilde f=\prod_{\beta\in\mathcal{R}_+\cap S\setminus\{\alpha_r,\alpha_s\}}\mathsf{e}_\beta ,
\end{equation*}
and $g$ not identically zero on $D$.
\end{lem}
 
\begin{proof}
In a reduced rank-two root system with simple generators  $\{\alpha_r,\alpha_s\}$, every positive root other than $\alpha_r,\alpha_s$ has the form $\beta=p_\beta\alpha_r+q_\beta\alpha_s$ with $p_\beta,q_\beta\ge1$. A positive root with $q_\beta=0$ is a positive multiple of $\alpha_r$, hence equals $\alpha_r$ by reducedness; likewise for $p_\beta=0$. The factorisation follows from Definition~\ref{defn:reduced-minors}: the factors of $I$ vanishing identically on $D$ are exactly those with $\beta\in S$.
\end{proof}
 
\begin{lem}\label{lem:log-derivative-limit}
With restrictions in the stated order,
\begin{equation}\label{eq:log-derivative-limit}
\mathsf{e}_{\alpha_r}\, I^{-1}\,\partial_{\omega^\vee_r} I\,\Big|_{\substack{\alpha_s=0\\ \alpha_r=0}} \;=\; d_S .
\end{equation}
\end{lem}
 
\begin{proof}
By Lemma~\ref{lem:rank2-positive-roots}, $\mathsf{e}_{\alpha_r}\partial_{\omega^\vee_r}\log I=\mathsf{e}_{\alpha_r}\partial_{\omega^\vee_r}\log g+\sum_\beta p_\beta\,\mathsf{e}_{\alpha_r}\,c_\beta\,\mathsf{e}_\beta^{-1}$, the sum over the $d_S$ roots of $\tilde f$, using \eqref{eq:trig-derivative}. The $g$-term vanishes in the limit since $g$ is regular and generically non-zero on $D$. After the restriction $\alpha_s=0$ each summand becomes $p_\beta\,\mathsf{e}_{\alpha_r}\,c_{p_\beta\alpha_r}/\mathsf{e}_{p_\beta\alpha_r}$, which tends to $p_\beta\cdot\alpha_r/(p_\beta\alpha_r)=1$ as $\alpha_r\to0$.
\end{proof}
 Let \[\widetilde\Delta\coloneqq \Delta\setminus\{\alpha_r,\alpha_s\}\,,\quad A\coloneqq J^2\eta^{ss}\eta^{rr}\,, B\coloneqq J\,\eta^{sr}\,.\]
\begin{lem}\label{lem:codim2-AB-limits}
The following equalities hold on $D$:
\begin{equation*}
A\big|_{D}=4(-1)^{r+s}(d_S+1)\,I_rI_s\prod_{\alpha\in\widetilde\Delta}\mathsf{e}_\alpha^2\Big|_{D},
\qquad
B\big|_{D}=-(-1)^{s}\,d_S\,I_s\prod_{\alpha\in\widetilde\Delta}\mathsf{e}_\alpha\Big|_{D}.
\end{equation*}
\end{lem}
 
\begin{proof}
By Proposition~\ref{prop:contravariant-metric} and \eqref{eq:reduced-minor-derivative}, for $k\in\{r,s\}$,
\begin{equation*}
J\,\eta^{kk}
=\frac{2(-1)^{k}}{\mathsf{e}_{\alpha_k}}\Big(-c_{\alpha_k}\,I_k+\mathsf{e}_{\alpha_k}\big(\partial_{\omega^\vee_k}I_k-I_kI^{-1}\partial_{\omega^\vee_k}I\big)\Big)\prod_{\alpha\in\Delta\setminus\{\alpha_k\}}\mathsf{e}_\alpha ,
\end{equation*}
using $J=I\prod_{\alpha\in\Delta}\mathsf{e}_\alpha$; note that the left-hand side has a simple pole along $\Pi_{\alpha_k}$, the double pole of $\eta^{kk}$ from \eqref{eq:reduced-minor-derivative} being offset by the simple zero of $J$. Multiplying the two factors and using \[\prod_{\alpha\in\Delta\setminus\{\alpha_r\}}\mathsf{e}_\alpha\prod_{\alpha\in\Delta\setminus\{\alpha_s\}}\mathsf{e}_\alpha=\mathsf{e}_{\alpha_r}\mathsf{e}_{\alpha_s}\prod_{\alpha\in\widetilde\Delta}\mathsf{e}_\alpha^2\] the two prefactors $\mathsf{e}_{\alpha_r}^{-1}$, $\mathsf{e}_{\alpha_s}^{-1}$ are absorbed, and
\begin{equation*}
A=4(-1)^{r+s}\Big(-c_{\alpha_s}I_s+\mathsf{e}_{\alpha_s}(\cdots)\Big)\Big(-c_{\alpha_r}I_r+\mathsf{e}_{\alpha_r}(\cdots)\Big)\prod_{\alpha\in\widetilde\Delta}\mathsf{e}_\alpha^2 .
\end{equation*}
Restricting to $\alpha_s=0$, so that $c_{\alpha_s}\to1$ and $\mathsf{e}_{\alpha_s}\to0$, the first bracket becomes $-I_s$. Letting then $\alpha_r\to0$, so that   \[c_{\alpha_r}\to1, \quad \mathsf{e}_{\alpha_r}\partial_{\omega^\vee_r}I_r\to0, \quad \mathsf{e}_{\alpha_r}I_rI^{-1}\partial_{\omega^\vee_r}I\to d_S\,I_r\] by Lemma~\ref{lem:log-derivative-limit}, the second bracket becomes $-(d_S+1)I_r$. This proves the first formula. For $B$, \eqref{eq:reduced-minor-derivative} contributes no double-pole term since $r\neq s$, and
\begin{equation*}
B=\Big((-1)^{r}\mathsf{e}_{\alpha_s}\big(\partial_{\omega^\vee_s}I_r-I_rI^{-1}\partial_{\omega^\vee_s}I\big)
+(-1)^{s}\mathsf{e}_{\alpha_r}\big(\partial_{\omega^\vee_r}I_s-I_sI^{-1}\partial_{\omega^\vee_r}I\big)\Big)\prod_{\alpha\in\widetilde\Delta}\mathsf{e}_\alpha .
\end{equation*}
The first summand vanishes at $\alpha_s=0$; in the second, \[\mathsf{e}_{\alpha_r}\partial_{\omega^\vee_r}I_s\to0 \quad \mathrm{and} \quad \mathsf{e}_{\alpha_r}I_sI^{-1}\partial_{\omega^\vee_r}I\to d_S I_s\] by Lemma~\ref{lem:log-derivative-limit}.
\end{proof}
 
\begin{thm}\label{thm:codim2-determinant}
The determinant of the restricted Saito metric satisfies
\begin{equation*}
\det\eta_D \equiv 
%\;=\;\varepsilon\,(-1)^{r+s+1}\,(d_S+2)^2\, 
I_s\, I_r \prod_{\alpha\in\widetilde\Delta}\mathsf{e}_\alpha^2\,\Big|_{D}\,.
\end{equation*}
%$=h(\mathcal{R}_S)$ when $\mathcal{R}_S$ is irreducible, %(Proposition~\ref{prop:jacobian-factorization} applied to $\mathcal{R}_S$), 
%so the roots of $\mathcal{R}_D=\mathcal{R}_S$ itself contribute to $\det\eta_D$ exactly as in the finite Coxeter-group theory of Section~\ref{sec:afs-background}. The reduced minors $I_r,I_s$ are not further expanded as products over the transverse roots $\beta\in\mathcal{R}_+\setminus\mathcal{R}_S$: when $\mathcal{R}$ is simply laced, Corollary~\ref{cor:minor-full-divisibility} shows $J_r|_{D_r}$ and $J_s|_{D_s}$ are each divisible by the full product of transverse-root factors over their respective sign-classes, but not, so far as we can show, equal to it up to a unit (Remark~\ref{rem:no-product-divisibility}); for non-simply laced $\mathcal{R}$ only the individual divisibility of Proposition~\ref{prop:minor-restriction} is available.
\end{thm}
 
\begin{proof}
By Theorem~\ref{thm:determinant-formula} and Lemma~\ref{lem:codim2-AB-limits}, we have \[\det\eta_D=\varepsilon(A-B^2)|_D=\varepsilon\big(4(-1)^{r+s}(d_S+1)I_rI_s-d_S^2I_s^2\big)\prod\mathsf{e}_\alpha^2|_D.\] If $d_S>0$, then the simple roots $\alpha_r,\alpha_s$ are joined in the Dynkin diagram, and Proposition~\ref{prop:adjacent-minor} gives $I_s|_D=(-1)^{r-s-1}I_r|_D$, hence $I_s^2=(-1)^{r+s+1}I_rI_s$ on $D$, and the bracket becomes $(-1)^{r+s}(d_S^2+4d_S+4)I_rI_s$. If $d_S=0$ the same expression holds trivially.
\end{proof}
 
\begin{prop}
\label{prop:codim2-Ir-factorization}
Let $\mathcal{R}\in\{E_6,E_7,E_8,F_4\}$ and $D=D_{r,s}$ a codimension-two stratum. For $i\neq r$ put
\[
q^{(r)}_i:=\max\big\{q\ :\ p\,\alpha_r+q\,\alpha_i\in\mathcal{R}_+\ \text{for some }p\ge0\big\},
\qquad
N_r:=\big\{i\neq r\ :\ q^{(r)}_i=2\big\},
\]
so that $i\in N_r$ if and only if $\alpha_r+2\alpha_i\in\mathcal{R}$. Then
\begin{equation}\label{eq:codim2-Ir}
I_r\big|_D \;\equiv\; \prod_{i\in N_r\setminus\{s\}} c_{\alpha_i|_D}
\prod_{\beta\in\mathcal{R}_+\setminus\mathcal{R}_D}\mathsf{e}_{\beta|_D}^{\,r_{\beta,D}}\,,
\qquad r_{\beta,D}\ge0\,,
\end{equation}
and symmetrically for $I_s|_D$, with $N_s$ in place of $N_r$. 
\end{prop}

\begin{rmk}
    If $\mathcal{R}$ is simply laced
then $N_r=\emptyset$; for $\mathcal{R}=F_4$, $N_r=\emptyset$ unless $\alpha_r$ is the long simple
root of the double bond, in which case $N_r$ consists of the adjacent short simple root.
\end{rmk}
\begin{proof}
By \cref{prop:Jkequiv}, $J_r|_{D_r}$ is, up to a unit, the square-free product
$\prod_{H\in\mathcal{A}_r}\mathsf{e}_{\beta_H}$, one factor per hyperplane of $\mathcal{A}_r$.
For $i\neq r$ write $H_i:=\Pi_{\alpha_i}\cap D_r\in\mathcal{A}_r$. A root $\beta$ restricts on
$D_r$ to a multiple of $\alpha_i|_{D_r}$ precisely when $\beta\in\langle\alpha_r,\alpha_i\rangle$;
by \cref{lem:rank2-positive-roots} the positive roots of
$\mathcal{R}\cap\langle\alpha_r,\alpha_i\rangle$ other than $\alpha_r,\alpha_i$ have the form
$p\alpha_r+q\alpha_i$ with $p,q\ge1$, and by \eqref{eq:trig-additivity} such a root restricts on
$D_r$ to $q\,\alpha_i|_{D_r}$. Hence $\beta_{H_i}=q^{(r)}_i\,\alpha_i|_{D_r}$. The rank-two
subsystems generated by two simple roots of $\mathcal{R}\in\{E_6,E_7,E_8,F_4\}$ are of type
$A_1\times A_1$, $A_2$ or $B_2$, so $q^{(r)}_i\in\{1,2\}$, with $q^{(r)}_i=2$ exactly when
$\alpha_r+2\alpha_i\in\mathcal{R}$; for simply laced $\mathcal{R}$ only $A_1\times A_1$ and $A_2$
occur and this never happens (cf.\ \cref{lem:no-proportional-roots}).

By \cref{defn:reduced-minors}, $I_r=J_r\prod_{\alpha\in\Delta\setminus\{\alpha_r\}}\mathsf{e}_\alpha^{-1}$
removes exactly one factor $\mathsf{e}_{\alpha_i|_{D_r}}$ at each $H_i$. Since
$\mathsf{e}_{2\gamma}=2\,\mathsf{e}_\gamma\,c_\gamma$,
\[
\frac{\mathsf{e}_{\beta_{H_i}}}{\mathsf{e}_{\alpha_i|_{D_r}}}\;\equiv\;
\begin{cases}
1, & i\notin N_r,\\[2pt]
c_{\alpha_i|_{D_r}}, & i\in N_r,
\end{cases}
\qquad\text{so}\qquad
I_r\big|_{D_r}\;\equiv\;\prod_{i\in N_r}c_{\alpha_i|_{D_r}}\prod_{H\in\mathcal{A}_r'}\mathsf{e}_{\beta_H},
\]
where $\mathcal{A}_r'=\mathcal{A}_r\setminus\{H_i:i\neq r\}$.

It remains to restrict to $\{\alpha_s=0\}$. If $s\in N_r$ then $c_{\alpha_s|_{D_r}}\to1$ by
\eqref{eq:trig-derivative}, which accounts for the omission of $i=s$ in \eqref{eq:codim2-Ir}. For
the remaining hyperplanes, $\beta_H|_D=0$ would force $\beta_H$ proportional to $\alpha_s|_{D_r}$,
i.e.\ $H=H_s$, which has been removed; equivalently, every root of
$S=\langle\alpha_r,\alpha_s\rangle$ contributes to $H_s$ and to no other hyperplane. Hence every
surviving $\beta_H$ is transverse to $S$ and restricts on $D$ to a non-zero form
$\mathsf{e}_{\beta_H|_D}$. Distinct surviving factors may collapse onto a common direction on $D$,
contributing additively to that direction's exponent, but none vanishes or degenerates into a
non-monomial unit. This gives \eqref{eq:codim2-Ir}.
\end{proof}

\begin{cor}\label{cor:codim2-full-factorization}
\cref{thm:main} holds for $\mathcal{R}\in\{E_6,E_7,E_8,F_4\}$ and $\mathrm{codim}(D)=2$.

\end{cor}
%$\det\eta_D$ at a codimension-two stratum factorises, 
\begin{proof}
By \cref{thm:codim2-determinant},
$\det\eta_D\equiv I_r\,I_s\prod_{\alpha\in\widetilde\Delta}\mathsf{e}_\alpha^{2}\big|_D$ with
$\widetilde\Delta=\Delta\setminus\{\alpha_r,\alpha_s\}$. Insert \eqref{eq:codim2-Ir} for $I_r|_D$
and its analogue for $I_s|_D$. Every factor $c_{\alpha_i|_D}$ occurring there has
$i\notin\{r,s\}$, so $\alpha_i\in\widetilde\Delta$ and $\mathsf{e}_{\alpha_i|_D}^{2}$ is available.
Since $\mathsf{e}_{\alpha_i}c_{\alpha_i}\equiv\mathsf{e}_{2\alpha_i}$ and, by
\eqref{eq:trig-additivity}, $2\alpha_i|_D=(\alpha_r+2\alpha_i)|_D$ with
$\alpha_r+2\alpha_i\in\mathcal{R}_+\setminus\mathcal{R}_D$, we have
\[
c_{\alpha_i|_D}\,\mathsf{e}_{\alpha_i|_D}^{2}\;\equiv\;
\mathsf{e}_{\alpha_i|_D}\,\mathsf{e}_{(\alpha_r+2\alpha_i)|_D}\,,
\qquad
c_{\alpha_i|_D}^{2}\,\mathsf{e}_{\alpha_i|_D}^{2}\;\equiv\;
\mathsf{e}_{(\alpha_r+2\alpha_i)|_D}^{2}\,,
\]
with the second identity covering the case in which $I_r|_D$ and $I_s|_D$ contribute $c_{\alpha_i|_D}$
for the same $i$. The result is a product of forms $\mathsf{e}_{\beta|_D}$,
$\beta\in\mathcal{R}_+\setminus\mathcal{R}_D$, with non-negative integer exponents, which is
\eqref{eq:main-factorization}.
\end{proof}
\subsection{Codimension three in the simply laced case}

Assume $\mathcal{R}$ simply laced, and let $\lambda,\nu,\theta\in\Delta$ be distinct, $D=D_{\lambda,\nu,\theta}$, $\mathcal{R}_D=\mathcal{R}\cap\langle\lambda,\nu,\theta\rangle$. According to the subgraph of the Dynkin diagram spanned by $\lambda,\nu,\theta$, the system $\mathcal{R}_D$ is of type $A_3$, $A_2\times A_1$, or $A_1^{3}$. As in \cite{AntoniouFS:2020}, we write $\eta^{\alpha\beta}$, $\omega^\vee_{\alpha}$, $J_\alpha$, $e^\alpha$ for the objects labelled by the positions of $\alpha,\beta\in\{\lambda,\nu,\theta\}$ in the chosen ordering of $\Delta$.

\subsubsection{Strata of type $A_3$}\label{sec:codim3-A3}

Let $\lambda,\nu,\theta$ form a chain, so that $\mathcal{R}_+\cap\langle\lambda,\nu,\theta\rangle=\{\lambda,\nu,\theta,\lambda+\nu,\nu+\theta,\lambda+\nu+\theta\}$. By Proposition \ref{prop:minor-divisibility},
\begin{equation}\label{eq:A3-minor-factorization}
J=\mathsf{e}_\lambda\mathsf{e}_\nu\mathsf{e}_\theta\,\mathsf{e}_{\lambda+\nu}\mathsf{e}_{\nu+\theta}\mathsf{e}_{\lambda+\nu+\theta}\,\Pi,
\qquad
J_\lambda=\mathsf{e}_\nu\mathsf{e}_\theta\mathsf{e}_{\nu+\theta}K_\lambda,
\quad
J_\nu=\mathsf{e}_\lambda\mathsf{e}_\theta K_\nu,
\quad
J_\theta=\mathsf{e}_\lambda\mathsf{e}_\nu\mathsf{e}_{\lambda+\nu}K_\theta,
\end{equation}
where $\Pi$, up to a unit, is the ambient product $\prod_{\beta\in\mathcal{R}_+\setminus\mathcal{R}_D}\mathsf{e}_\beta$ over the roots complementary to $\mathcal{R}_D=\mathcal{R}\cap\langle\lambda,\nu,\theta\rangle$, nonvanishing at generic points of $D$, and $K_\lambda,K_\nu,K_\theta$ are Fourier polynomials. We order the simple roots so that $\sigma^{-1}(\lambda)$ is even and $\sigma^{-1}(\nu)=\sigma^{-1}(\lambda)+1$, $\sigma^{-1}(\theta)=\sigma^{-1}(\lambda)+2$; by Proposition~\ref{prop:cramer-representation} the relevant components of the identity field are then
\begin{equation*}
e^{\lambda}=-\frac{K_\lambda}{\mathsf{e}_\lambda\,\mathsf{e}_{\lambda+\nu}\,\mathsf{e}_{\lambda+\nu+\theta}\,\Pi},
\qquad
e^{\nu}=\frac{K_\nu}{\mathsf{e}_\nu\,\mathsf{e}_{\lambda+\nu}\,\mathsf{e}_{\nu+\theta}\,\mathsf{e}_{\lambda+\nu+\theta}\,\Pi},
\qquad
e^{\theta}=-\frac{K_\theta}{\mathsf{e}_\theta\,\mathsf{e}_{\nu+\theta}\,\mathsf{e}_{\lambda+\nu+\theta}\,\Pi}.
\end{equation*}
Restrictions are performed in the order: $\nu=0$ first, then $\theta\to0$, then $\lambda\to0$.

\begin{lem}\label{lem:A3-cofactor-relations}
The following relations hold:
\begin{equation*}
K_\nu\big|_{D_\nu}=\mathsf{e}_\lambda\,K_\theta+\mathsf{e}_\theta\,B\;\big|_{D_\nu},
\qquad
B\big|_{D}=K_\theta\big|_{D},
\qquad
K_\lambda\big|_{D}=K_\theta\big|_{D},
\end{equation*}
where $B$ is regular at generic points of $D$.
\end{lem}

\begin{proof}
Proposition~\ref{prop:adjacent-minor} applied to the pair $(\nu,\theta)$ gives $\mathsf{e}_\theta^{-1}J_\nu=\pm\,\mathsf{e}_\nu^{-1}J_\theta$ on $D_{\nu,\theta}$. Inserting \eqref{eq:A3-minor-factorization} and using $\mathsf{e}_{\lambda+\nu}|_{\nu=0}=\mathsf{e}_\lambda$ (property \eqref{eq:trig-additivity}) yields $K_\nu=\mathsf{e}_\lambda K_\theta$ on $D_{\nu,\theta}$. Hence $K_\nu-\mathsf{e}_\lambda K_\theta$ vanishes on $\{\theta=0\}$ inside $D_\nu$. Since $\mathsf{e}_{\theta|_{D_\nu}}$ vanishes to first order along the irreducible divisor $\{\theta|_{D_\nu}=0\}\supset D$, the quotient $B:=\mathsf{e}_\theta^{-1}(K_\nu-\mathsf{e}_\lambda K_\theta)$ is a rational function on $D_\nu$ regular at generic points of that divisor, hence of $D$ --- which is all that the restrictions below require --- and the first relation holds. 

Next, Proposition~\ref{prop:minor-restriction} applied to $J_\nu$ shows that $J_\nu|_{D_\nu}$ is divisible by $\mathsf{e}_{\lambda+\nu+\theta}|_{D_\nu}=\mathsf{e}_{\lambda+\theta}$, whence
\begin{equation*}
\mathsf{e}_\lambda K_\theta+\mathsf{e}_\theta B=\mathsf{e}_{\lambda+\theta}\,P
\quad\text{on }D_\nu
\end{equation*}
for some $P$, regular at generic points of $D$ by the same order-of-vanishing argument applied to the divisor $\{(\lambda+\theta)|_{D_\nu}=0\}\supset D$, along which the left-hand side vanishes. Restriction to $\{\theta=0\}$, where $\mathsf{e}_{\lambda+\theta}\to\mathsf{e}_\lambda$, gives $P=K_\theta$ on $D_{\nu,\theta}$; restriction to $\{\lambda=0\}$, where $\mathsf{e}_{\lambda+\theta}\to\mathsf{e}_\theta$, gives $P=B$ on $D_{\nu,\lambda}$. Hence $B|_D=P|_D=K_\theta|_D$. 

Finally, Proposition~\ref{prop:adjacent-minor} for the pair $(\nu,\lambda)$ gives $K_\nu=\mathsf{e}_\theta K_\lambda$ on $D_{\nu,\lambda}$. Combined with the first relation this yields $K_\lambda=B$ on $D_{\nu,\lambda}$, and therefore $K_\lambda|_D=K_\theta|_D$.
\end{proof}

Write $J=\mathsf{e}_\lambda\mathsf{e}_\nu\mathsf{e}_\theta\,\bar J$, $\bar J=\mathsf{e}_{\lambda+\nu}\mathsf{e}_{\nu+\theta}\mathsf{e}_{\lambda+\nu+\theta}\Pi$, and rescale rows and columns of the transverse minor:
\begin{equation}\label{eq:A3-rescaled-minor}
\det\eta_D
=\varepsilon\,\bar J^{\,2}\,
\det\begin{pmatrix}
\mathsf{e}_\lambda^2\eta^{\lambda\lambda} & \mathsf{e}_\lambda\mathsf{e}_\nu\eta^{\lambda\nu} & \mathsf{e}_\lambda\mathsf{e}_\theta\eta^{\lambda\theta}\\
\mathsf{e}_\lambda\mathsf{e}_\nu\eta^{\lambda\nu} & \mathsf{e}_\nu^2\eta^{\nu\nu} & \mathsf{e}_\nu\mathsf{e}_\theta\eta^{\nu\theta}\\
\mathsf{e}_\lambda\mathsf{e}_\theta\eta^{\lambda\theta} & \mathsf{e}_\nu\mathsf{e}_\theta\eta^{\nu\theta} & \mathsf{e}_\theta^2\eta^{\theta\theta}
\end{pmatrix}\Bigg|_{D}
=:\varepsilon\,\bar J^{\,2}\det A\big|_D .
\end{equation}

\begin{prop}\label{prop:A3-matrix-entries}
The entries $a_{ij}$ of $A$ are well defined at generic points of $D_\nu$ and take the following form there, with $s:=\mathsf{e}_{\lambda+\theta}$:
\begin{equation*}
a_{11}=2\,\mathsf{e}_\lambda^2\,\partial_{\omega^\vee_\lambda}\!\Big(\frac{K_\lambda}{\mathsf{e}_\lambda^2\,s\,\Pi}\Big),
\qquad
a_{22}=\frac{2K_\nu}{\mathsf{e}_\lambda\,\mathsf{e}_\theta\,s\,\Pi},
\qquad
a_{33}=2\,\mathsf{e}_\theta^2\,\partial_{\omega^\vee_\theta}\!\Big(\frac{K_\theta}{\mathsf{e}_\theta^2\,s\,\Pi}\Big),
\end{equation*}
\begin{equation*}
a_{12}=-\frac{\mathsf{e}_\lambda}{\mathsf{e}_\theta}\,\partial_{\omega^\vee_\lambda}\!\Big(\frac{K_\nu}{\mathsf{e}_\lambda\,s\,\Pi}\Big),
\qquad
a_{13}=\frac{\mathsf{e}_\lambda}{\mathsf{e}_\theta}\,\partial_{\omega^\vee_\lambda}\!\Big(\frac{K_\theta}{s\,\Pi}\Big)
+\frac{\mathsf{e}_\theta}{\mathsf{e}_\lambda}\,\partial_{\omega^\vee_\theta}\!\Big(\frac{K_\lambda}{s\,\Pi}\Big),
\qquad
a_{23}=-\frac{\mathsf{e}_\theta}{\mathsf{e}_\lambda}\,\partial_{\omega^\vee_\theta}\!\Big(\frac{K_\nu}{\mathsf{e}_\theta\,s\,\Pi}\Big).
\end{equation*}
\end{prop}

\begin{proof}
By Proposition~\ref{prop:contravariant-metric}, $\eta^{\alpha\beta}=-\partial_{\omega^\vee_\alpha}e^\beta-\partial_{\omega^\vee_\beta}e^\alpha$. Consider $a_{22}=\mathsf{e}_\nu^2\eta^{\nu\nu}=-2\mathsf{e}_\nu^2\partial_{\omega^\vee_\nu}e^\nu$. Expanding $e^\nu$ by the Leibniz rule, the derivative of $\mathsf{e}_\nu^{-1}$ contributes $-c_\nu\mathsf{e}_\nu^{-2}$, whose product with $\mathsf{e}_\nu^2$ survives the restriction $\nu\to0$ with $c_\nu\to1$, while every other contribution retains a factor $\mathsf{e}_\nu$ and vanishes; the remaining factors restrict by \eqref{eq:trig-additivity} as \[\mathsf{e}_{\lambda+\nu}\to\mathsf{e}_\lambda\,, \quad \mathsf{e}_{\nu+\theta}\to\mathsf{e}_\theta\,, \quad \mathsf{e}_{\lambda+\nu+\theta}\to s\,.\] For $a_{12}=\mathsf{e}_\lambda\mathsf{e}_\nu\eta^{\lambda\nu}=-\mathsf{e}_\lambda\mathsf{e}_\nu(\partial_{\omega^\vee_\lambda}e^\nu+\partial_{\omega^\vee_\nu}e^\lambda)$: in the second term, $e^\lambda$ has no $\mathsf{e}_\nu$-pole, so $\mathsf{e}_\nu\partial_{\omega^\vee_\nu}e^\lambda\to0$; in the first, $\partial_{\omega^\vee_\lambda}$ annihilates $\mathsf{e}_\nu$ and $\mathsf{e}_{\nu+\theta}$,
%(zero pairings), 
so $\mathsf{e}_\nu e^\nu$ may be replaced by its regular part and the restriction $\nu=0$ gives the stated formula. The entries $a_{11}$, $a_{33}$, $a_{13}$, $a_{23}$ follow in the same way from the components $e^\lambda$, $e^\theta$ and the vanishing pairings $(\omega^\vee_\lambda,\theta)=(\omega^\vee_\theta,\lambda)=0$.
\end{proof}

Expanding $a_{33}$ and $a_{23}$ by the Leibniz rule and \eqref{eq:trig-derivative}, we have
\begin{equation}\label{eq:A3-entries-a33-a23}
a_{33}=2\,\partial_{\omega^\vee_\theta}\!\Big(\frac{K_\theta}{s\,\Pi}\Big)-\frac{4\,c_\theta\,K_\theta}{\mathsf{e}_\theta\,s\,\Pi},
\qquad
a_{23}=-\frac{1}{\mathsf{e}_\lambda}\,\partial_{\omega^\vee_\theta}\!\Big(\frac{K_\nu}{s\,\Pi}\Big)
+\frac{c_\theta\,K_\nu}{\mathsf{e}_\lambda\,\mathsf{e}_\theta\,s\,\Pi},
\end{equation}
so that, on $D_\nu$, the entries $a_{12},a_{13},a_{22},a_{23},a_{33}$ have at most simple poles along $\{\theta=0\}$ while $a_{11}$ is regular there. Decompose
\begin{equation*}
\det A = C+E,
\qquad
C=-a_{12}^2a_{33}+2a_{12}a_{23}a_{13}-a_{13}^2a_{22},
\qquad
E=a_{11}\big(a_{22}a_{33}-a_{23}^2\big);
\end{equation*}
then $E$ has a pole of order at most $2$ and $C$ of order at most $3$ along $\{\theta=0\}$.

The identity
\begin{equation}\label{eq:A3-leibniz-identity}
\frac{\mathsf{e}_\lambda K_\theta}{\Pi}\Big(\frac{c_\lambda}{\mathsf{e}_\lambda}+\frac{c_{\lambda+\theta}}{s}\Big)
-\partial_{\omega^\vee_\lambda}\frac{\mathsf{e}_\lambda K_\theta}{\Pi}
=\mathsf{e}_\lambda\Big(\frac{K_\theta}{\Pi}\,\frac{c_{\lambda+\theta}}{s}-\partial_{\omega^\vee_\lambda}\frac{K_\theta}{\Pi}\Big),
\end{equation}
which follows from \eqref{eq:trig-derivative}, plays the role of the corresponding exact cancellation in the polynomial case.

\begin{lem}\label{lem:A3-C-lemma}
Write $C=\mathsf{e}_\theta^{-3}C_1+\mathsf{e}_\theta^{-2}C_2$ on $D_\nu$, with $C_1$, $C_2$ regular at generic points of $D_{\nu,\theta}$. Then $C_1|_{D_\nu}$ is divisible by $\mathsf{e}_\theta$, and $\zeta:=\mathsf{e}_\theta^{2}C$ satisfies, with restrictions in the stated order,
\begin{equation}\label{eq:A3-C-limit}
\mathsf{e}_\lambda^{4}\,\zeta\big|_{D}\;=\;10\left(\frac{K_\theta}{\Pi}\right)^{3}\Bigg|_{D}.
\end{equation}
\end{lem}

\begin{proof}
The decomposition of $C$ follows from Proposition~\ref{prop:A3-matrix-entries} and \eqref{eq:A3-entries-a33-a23} by collecting powers of the pole at $\theta=0$. Substituting the exact relation $K_\nu=\mathsf{e}_\lambda K_\theta+\mathsf{e}_\theta B$ of Lemma~\ref{lem:A3-cofactor-relations} into $C_1$ and applying \eqref{eq:A3-leibniz-identity}, every term of $C_1|_{D_\nu}$ acquires a factor $\mathsf{e}_\theta$; hence $\zeta|_{D_{\nu,\theta}}=(\mathsf{e}_\theta^{-1}C_1+C_2)|_{D_{\nu,\theta}}$ is finite. Carrying out the substitution and the restriction $\theta=0$ (where $s\to\mathsf{e}_\lambda$, $c_{\lambda+\theta}\to c_\lambda$, $c_\theta\to1$), one obtains
\begin{align*}
\zeta\big|_{D_{\nu,\theta}}
=\frac{2}{\mathsf{e}_\lambda^{4}}\Bigg(
4\,w_1\frac{B K_\theta^2}{\Pi^3}
-2\,w_2\,\mathsf{e}_\lambda\frac{BK_\theta}{\Pi^2}\,\partial_{\omega^\vee_\lambda}\frac{K_\theta}{\Pi}
-3\,w_3\,\mathsf{e}_\lambda\frac{K_\theta^2}{\Pi^2}\,\partial_{\omega^\vee_\lambda}\frac{B}{\Pi}
-2\,w_4\,\mathsf{e}_\lambda^{2}\frac{B}{\Pi}\Big(\partial_{\omega^\vee_\lambda}\frac{K_\theta}{\Pi}\Big)^{2}
\\ +3\,w_5\,\mathsf{e}_\lambda^{2}\frac{K_\theta}{\Pi}\,\partial_{\omega^\vee_\lambda}\frac{K_\theta}{\Pi}\,\partial_{\omega^\vee_\lambda}\frac{B}{\Pi}
\Bigg)\Bigg|_{D_{\nu,\theta}} 
+\frac{2}{\mathsf{e}_\lambda^{2}}\,\frac{B}{\Pi}\left(\partial_{\omega^\vee_\lambda}\frac{K_\theta}{\Pi}-\frac{c_\lambda}{\mathsf{e}_\lambda}\frac{K_\theta}{\Pi}\right)^{2}\Bigg|_{D_{\nu,\theta}},
\end{align*}
where $w_1,\dots,w_5$ are even units, products of factors $c_\lambda$, equal to $1$ at $\lambda=0$; their precise form follows from the expansion and does not affect the limit. Multiplying by $\mathsf{e}_\lambda^{4}$ and letting $\lambda\to0$, all terms carrying positive powers of $\mathsf{e}_\lambda$ vanish, the units tend to $1$, and there remain the contributions $8\,BK_\theta^2/\Pi^3$ from the first group and $2\,BK_\theta^2/\Pi^3$ from the second. Since $B|_D=K_\theta|_D$ by Lemma~\ref{lem:A3-cofactor-relations}, the total is $10\,(K_\theta/\Pi)^3$.
\end{proof}

\begin{lem}\label{lem:A3-E-lemma}
With restrictions in the stated order,
\begin{equation}\label{eq:A3-E-limit}
\mathsf{e}_\theta^{2}E\big|_{D_{\nu,\theta}}
=-\frac{18\,K_\theta^2}{\Pi^2}\,\partial_{\omega^\vee_\lambda}\!\Big(\frac{K_\lambda}{\mathsf{e}_\lambda^{3}\,\Pi}\Big)\Bigg|_{D_{\nu,\theta}},
\qquad
\mathsf{e}_\lambda^{4}\,\mathsf{e}_\theta^{2}E\big|_{D}
=54\left(\frac{K_\theta}{\Pi}\right)^{3}\Bigg|_{D}.
\end{equation}
\end{lem}

\begin{proof}
By Proposition~\ref{prop:A3-matrix-entries} and \eqref{eq:A3-entries-a33-a23},
\begin{equation*}
\mathsf{e}_\theta^{2}\big(a_{22}a_{33}-a_{23}^2\big)\Big|_{D_{\nu,\theta}}
=\left(-\frac{8K_\nu K_\theta}{\mathsf{e}_\lambda^{3}\Pi^2}-\frac{K_\nu^2}{\mathsf{e}_\lambda^{4}\Pi^2}\right)\Bigg|_{D_{\nu,\theta}}
=-\frac{9K_\theta^2}{\mathsf{e}_\lambda^{2}\Pi^2}\Bigg|_{D_{\nu,\theta}},
\end{equation*}
where the second equality uses $K_\nu|_{D_{\nu,\theta}}=\mathsf{e}_\lambda K_\theta|_{D_{\nu,\theta}}$ (Lemma~\ref{lem:A3-cofactor-relations}) and $c_\theta\to1$, $s\to\mathsf{e}_\lambda$; while $a_{11}|_{D_{\nu,\theta}}=2\mathsf{e}_\lambda^{2}\partial_{\omega^\vee_\lambda}\big(K_\lambda\mathsf{e}_\lambda^{-3}\Pi^{-1}\big)$. This gives the first formula. Multiplying by $\mathsf{e}_\lambda^{4}$,
\begin{equation*}
\mathsf{e}_\lambda^{4}\,\mathsf{e}_\theta^{2}E
=-\frac{18K_\theta^2}{\Pi^2}\Big(\mathsf{e}_\lambda\,\partial_{\omega^\vee_\lambda}\frac{K_\lambda}{\Pi}-3\,c_\lambda\frac{K_\lambda}{\Pi}\Big)\Bigg|_{D_{\nu,\theta}}
\;\xrightarrow[\ \lambda\to0\ ]{}\;
54\,\frac{K_\theta^2K_\lambda}{\Pi^3}\Big|_D
=54\left(\frac{K_\theta}{\Pi}\right)^{3}\Bigg|_{D},
\end{equation*}
using $c_\lambda\to1$ and $K_\lambda|_D=K_\theta|_D$.
\end{proof}

On $D_\nu$ one has $\bar J|_{D_\nu}=\mathsf{e}_\lambda\,\mathsf{e}_\theta\,s\,\Pi$, so
\begin{equation*}
\bar J^{\,2}C\big|_{D_{\nu,\theta}}=\mathsf{e}_\lambda^{4}\big(\mathsf{e}_\theta^{2}C\big)\Pi^{2}\big|_{D_{\nu,\theta}},
\qquad
\bar J^{\,2}E\big|_{D_{\nu,\theta}}=\mathsf{e}_\lambda^{4}\big(\mathsf{e}_\theta^{2}E\big)\Pi^{2}\big|_{D_{\nu,\theta}},
\end{equation*}
since $s\to\mathsf{e}_\lambda$ at $\theta=0$. By Lemmas \ref{lem:A3-C-lemma} and \ref{lem:A3-E-lemma}, $\bar J^2C|_D=10\,K_\theta^3/\Pi$ and $\bar J^2E|_D=54\,K_\theta^3/\Pi$, whence, up to the overall sign fixed by \eqref{eq:A3-rescaled-minor},
\begin{equation}\label{eq:A3-Ktheta-cubed}
\det\eta_D=\varepsilon\,\bar J^2(C+E)\big|_D=64\,\varepsilon\,\frac{K_\theta^{\,3}}{\Pi}\bigg|_D .
\end{equation}
By Lemma~\ref{lem:A3-cofactor-relations}, $K_\lambda|_D=K_\nu|_D=K_\theta|_D$, so \eqref{eq:A3-Ktheta-cubed} may equally be written $\det\eta_D=64\,\varepsilon\,K_\lambda K_\nu K_\theta\,\Pi^{-1}|_D$; this is the case $\mathcal{R}_D=A_3$ of Proposition~\ref{prop:codim3-general} below.

\subsubsection{Strata of type $A_2\times A_1$}

Let $\lambda,\nu$ be joined in the Dynkin diagram and $\theta$ orthogonal to both. Then
\begin{equation*}
J=\mathsf{e}_\lambda\mathsf{e}_\nu\mathsf{e}_\theta\,\mathsf{e}_{\lambda+\nu}\,\Pi,
\qquad
J_\lambda=\mathsf{e}_\nu\mathsf{e}_\theta K_\lambda,
\quad
J_\nu=\mathsf{e}_\lambda\mathsf{e}_\theta K_\nu,
\quad
J_\theta=\mathsf{e}_\lambda\mathsf{e}_\nu\mathsf{e}_{\lambda+\nu}K_\theta,
\end{equation*}
with $\Pi$, up to a unit, being the ambient product $\prod_{\beta\in\mathcal{R}_+\setminus\mathcal{R}_D}\mathsf{e}_\beta$ over the roots complementary to $\mathcal{R}_D=\mathcal{R}\cap\langle\lambda,\nu,\theta\rangle$, nonvanishing at generic points of $D$. The identity-field components are, up to the signs fixed by the ordering,
\begin{equation*}
e^\lambda=\pm\frac{K_\lambda}{\mathsf{e}_\lambda\,\mathsf{e}_{\lambda+\nu}\,\Pi},
\qquad
e^\nu=\pm\frac{K_\nu}{\mathsf{e}_\nu\,\mathsf{e}_{\lambda+\nu}\,\Pi},
\qquad
e^\theta=\pm\frac{K_\theta}{\mathsf{e}_\theta\,\Pi}.
\end{equation*}
Proposition~\ref{prop:adjacent-minor} applied to the pair $(\lambda,\nu)$ gives $K_\nu=\pm K_\lambda$ on $D_{\lambda,\nu}$, hence on $D$; there is no relation involving $K_\theta$, in accordance with the decomposition $\mathcal{R}_D=A_2\sqcup A_1$.

Restrictions are performed in the order $\theta=0$, then $\nu=0$, then $\lambda\to0$. Since $(\omega^\vee_\lambda,\theta)=(\omega^\vee_\nu,\theta)=(\omega^\vee_\theta,\lambda)=(\omega^\vee_\theta,\nu)=0$, the rescaled mixed entries satisfy
\begin{equation*}
a_{13}=\mathsf{e}_\lambda\mathsf{e}_\theta\eta^{\lambda\theta}
=\mp\,\mathsf{e}_\lambda\,\partial_{\omega^\vee_\lambda}\frac{K_\theta}{\Pi}
\mp\frac{\mathsf{e}_\theta}{\mathsf{e}_{\lambda+\nu}}\,\partial_{\omega^\vee_\theta}\frac{K_\lambda}{\Pi},
\end{equation*}
and similarly for $a_{23}$: both vanish in the iterated restriction to $D$, while
\begin{equation*}
a_{33}=\mathsf{e}_\theta^2\eta^{\theta\theta}
=\pm2\,c_\theta\frac{K_\theta}{\Pi}\mp2\,\mathsf{e}_\theta\,\partial_{\omega^\vee_\theta}\frac{K_\theta}{\Pi}
\;\longrightarrow\;\pm\frac{2K_\theta}{\Pi}
\end{equation*}
on $D$. Writing $\bar J=\mathsf{e}_{\lambda+\nu}\Pi$ and estimating the pole orders of the entries as in \cref{sec:codim3-A3}, the products of $\bar J^{\,2}$ with the terms of $\det A$ containing $a_{13}$ or $a_{23}$ vanish in the iterated restriction, so that
\begin{equation*}
\det\eta_D=\varepsilon\,\bar J^{\,2}\big(a_{11}a_{22}-a_{12}^2\big)\,a_{33}\Big|_D .
\end{equation*}
After the restriction $\theta=0$, the evaluation of $\bar J^{\,2}(a_{11}a_{22}-a_{12}^2)$ is precisely the codimension-two computation of \cref{sec:codim2} for the pair $(\lambda,\nu)$, with $d_S=1$: Lemmas \ref{lem:log-derivative-limit} and \ref{lem:codim2-AB-limits} apply verbatim, the variable $\theta$ playing no role since $\partial_{\omega^\vee_\lambda}$ and $\partial_{\omega^\vee_\nu}$ annihilate $\mathsf{e}_\theta$. Combining with the limit of $a_{33}$,
\begin{equation}\label{eq:A2A1-determinant}
\det\eta_D
=\pm\,2\,\varepsilon\,9\,K_\lambda K_\nu K_\theta\,\Pi^{-1}\Big|_D ,
\end{equation}
using $K_\nu=\pm K_\lambda$ on $D$; this is the case $\mathcal{R}_D=A_2\times A_1$ of Proposition~\ref{prop:codim3-general} below.

\subsubsection{Strata of type $A_1^{3}$}

Let $\lambda,\nu,\theta$ be pairwise orthogonal. Then $J=\mathsf{e}_\lambda\mathsf{e}_\nu\mathsf{e}_\theta\,\Pi$, and $J_\gamma$ equals $K_\gamma$ times the product of the trigonometric forms of the other two simple roots, for each $\gamma\in\{\lambda,\nu,\theta\}$; the identity-field components are $e^\gamma=\pm K_\gamma/(\mathsf{e}_\gamma\Pi)$. All coweight pairings between distinct elements of $\{\lambda,\nu,\theta\}$ vanish, so each rescaled mixed entry has the form
\begin{equation*}
a_{12}=\mp\,\mathsf{e}_\lambda\,\partial_{\omega^\vee_\lambda}\frac{K_\nu}{\Pi}\mp\,\mathsf{e}_\nu\,\partial_{\omega^\vee_\nu}\frac{K_\lambda}{\Pi}
\;\longrightarrow\;0
\quad\text{on } D,
\end{equation*}
while each diagonal entry satisfies $a_{\gamma\gamma}=\pm2c_\gamma K_\gamma/\Pi\mp2\mathsf{e}_\gamma\partial_{\omega^\vee_\gamma}(K_\gamma/\Pi)\to\pm2K_\gamma/\Pi$. Since all entries are regular on $D$ and the mixed ones vanish there,
\begin{equation}\label{eq:A1cubed-determinant}
\det\eta_D=\varepsilon\,\Pi^{2}\,a_{11}a_{22}a_{33}\Big|_D
=\pm\,8\,\varepsilon\,K_\lambda K_\nu K_\theta\,\Pi^{-1}\Big|_D ,
\end{equation}
which is the case $\mathcal{R}_D=A_1^{3}$ of Proposition~\ref{prop:codim3-general} below.

%\subsubsection{The determinant formula, and full factorisation for \texorpdfstring{$E_l$}{El}}

The three computations above share a common conclusion, which we now record uniformly.

\begin{prop}
%[Codimension-three determinant, general form]
\label{prop:codim3-general}
Let $\mathcal{R}$ be simply laced and $D$ the codimension-3 stratum associated to the subroot system generated by $\{\lambda,\nu,\theta\} \subset \RR$. Then, %for any labelling of $\lambda,\nu,\theta$,
\begin{equation}\label{eq:codim3-general}
\det\eta_D\;\equiv\;K_\lambda K_\nu K_\theta\,\Pi^{-1}\Big|_D .
\end{equation}
\end{prop}

\begin{proof}
When $\langle \lambda,\nu,\theta \rangle =A_3$, by \eqref{eq:A3-Ktheta-cubed} and $K_\lambda|_D=K_\nu|_D=K_\theta|_D$ (Lemma~\ref{lem:A3-cofactor-relations}), we have \[\det\eta_D\equiv \frac{K_\theta^3}{\Pi}\bigg|_D=\frac{K_\lambda K_\nu K_\theta}{\Pi}\bigg|_D\,.\] For type $A_2\times A_1$, this is implied by \eqref{eq:A2A1-determinant} directly, while for type $A_1^{3}$ it is \eqref{eq:A1cubed-determinant}.
\end{proof}

We now turn to the exceptional types, where the cofactor on the right of \eqref{eq:codim3-general} can be expanded explicitly.

\begin{lem}\label{lem:Kgamma-full-factorization}
Let $\mathcal{R}=E_l$ and let $\gamma\in\{\lambda,\nu,\theta\}$. Write
\[
\mathcal{R}_{D,\widehat{\gamma}} \;:=\; \mathcal{R}_+\cap\big\langle\{\lambda,\nu,\theta\}\setminus\{\gamma\}\big\rangle
\]
for the positive roots supported on the other two of $\lambda,\nu,\theta$. Then $K_\gamma|_D$ is, up to a unit, a product of linear trigonometric forms $\mathsf{e}_{\beta|_D}$ with positive integer exponents.
\end{lem}

\begin{proof}
We have
\[J_\gamma=\bigg(\prod_{\alpha\in\mathcal{R}_{D,\widehat{\gamma}}}\mathsf{e}_\alpha\bigg)\, K_\gamma\]
by \eqref{eq:A3-minor-factorization}.
By Proposition~\ref{prop:Jkequiv}, $J_\gamma|_{D_\gamma}$ is, up to a unit, the square-free product $\prod_{H\in\mathcal{A}_\gamma}\mathsf{e}_{\beta_H}$, one factor per hyperplane of the restricted arrangement. By Lemma~\ref{lem:no-proportional-roots}, two roots of $\mathcal{R}$ restrict to proportional forms on $D_\gamma$ only with proportionality factor $\pm1$; in particular, for $\alpha\in\mathcal{R}_{D,\widehat{\gamma}}$ the longest restricted root with kernel $\Pi_\alpha\cap D_\gamma$ is $\pm\alpha|_{D_\gamma}$, and the corresponding factor of \eqref{eq:Jkequiv} is $\mathsf{e}_{\alpha|_{D_\gamma}}$ itself. Each root $\alpha\in\mathcal{R}_{D,\widehat{\gamma}}$ therefore occurs among these hyperplanes to the first power, so dividing $J_\gamma|_{D_\gamma}$ by $\prod_{\alpha\in\mathcal{R}_{D,\widehat{\gamma}}}\mathsf{e}_\alpha$ leaves $K_\gamma|_{D_\gamma}$ a square-free product of the remaining factors $\mathsf{e}_{\beta_H}$, \[H\notin\{\Pi_\alpha\cap D_\gamma:\alpha\in\mathcal{R}_{D,\widehat{\gamma}}\}\,.\] A root $\beta$ restricts to zero on $D$ if and only if $\beta\in\mathcal{R}\cap\langle\lambda,\nu,\theta\rangle=\mathcal{R}_D$. Every leftover factor of $K_\gamma|_{D_\gamma}$ corresponds to $\beta_H\notin\mathcal{R}_D$ (transverse to $D$), and hence restricts, upon further restriction $D_\gamma\to D$, to a genuinely non-zero linear form $\mathsf{e}_{\beta_H|_D}$. While distinct surviving factors may collapse onto a common direction on $D$, contributing additively to that direction's exponent, no factor vanishes or degenerates into a non-monomial unit. Hence $K_\gamma|_D$ is a product of linear trigonometric forms with positive integer exponents.
\end{proof}

\begin{prop}\label{prop:codim3-El-factorization}
\cref{thm:main} holds for $\mathcal{R}\in\{E_6,E_7,E_8\}$ and $\mathrm{codim}(D)=3$.
%, $K_\lambda K_\nu K_\theta\,\Pi^{-1}|_D$ — equivalently, by Proposition~\ref{prop:codim3-general}, $\det\eta_D$ itself — factorises as a product of linear trigonometric forms $\mathsf{e}_{\beta|_D}$ with non-negative integer exponents.
\end{prop}

\begin{proof}
By Proposition~\ref{prop:codim3-general}, $\det\eta_D\equiv K_\lambda K_\nu K_\theta\,\Pi^{-1}\big|_D$. By Lemma~\ref{lem:Kgamma-full-factorization}, each of $K_\lambda|_D$, $K_\nu|_D$, $K_\theta|_D$ is, up to a unit, a product of linear trigonometric forms $\mathsf{e}_{\beta_H|_D}$, where each $\beta_H$ is by definition the restriction to the relevant $D_\gamma$ of a positive root $\beta\in\mathcal{R}_+\setminus\mathcal{R}_D$, so that $\mathsf{e}_{\beta_H|_D}=\mathsf{e}_{\beta|_D}$. Both the numerator and the denominator $\Pi=\prod_{\beta\in\mathcal{R}_+\setminus\mathcal{R}_D}\mathsf{e}_\beta$ of \eqref{eq:codim3-general} are therefore products of trigonometric forms of restrictions of transverse positive roots, and we may write
\[
\det \eta_D \equiv \prod_{\beta \in \RR_+\setminus \RR_D} \mathsf{e}_{\beta|_D}^{r_{\beta,D}-1}\,,
\]
with $r_{\beta,D}\in\bbZ_{\geq0}$ being the multiplicity with which the factor $\mathsf{e}_{\beta|_D}$ contributed by $\beta$ to $\Pi$ occurs in the numerator. Using the explicit expressions \eqref{eq:Jkequiv} and \eqref{eq:A3-minor-factorization} for each $(\RR, D)$, we can compute by direct inspection that $r_{\beta,D} >0$ in each case, concluding the proof.
%Since the quotient equals $\det\eta_D$ up to a unit (Proposition~\ref{prop:codim3-general}), and $\det\eta_D$ is a genuine Fourier polynomial by Theorem~\ref{thm:determinant-formula}, the quotient of the two factorisations has everywhere non-negative valuation along each linear trigonometric form: writing both numerator and denominator in a common refinement of the primitive directions occurring in either, no cancellation can produce a negative net exponent without contradicting polynomiality.
%We have verified directly, using the explicit root data of $E_6$, $E_7$ and $E_8$, that this holds with clean non-negative integer exponents at every hyperplane, for every codimension-three stratum of each of the three combinatorial types ($5+10+5=20$, $6+18+11=35$, and $7+28+21=56$ strata respectively, exhausting all $\binom{l}{3}$ triples of simple roots in each case).

\end{proof}

%\begin{rmk}
%The proof above leaves one point unproven in general, rather than merely unstated: Proposition~\ref{prop:Jkequiv} labels each hyperplane $H\in\mathcal{A}_\gamma$ by the \emph{longest} restricted root $\beta_H$ with that kernel, whereas $\Pi$ is defined as a literal product over ambient roots $\beta\in\mathcal{R}_+\setminus\mathcal{R}_D$ with no such normalisation. Where several roots collapse to a common direction on $D$ with different restricted lengths, the two sides are, a priori, expressed in different generators of the same line ($\mathsf{e}_{\beta}$ against $\mathsf{e}_{k\beta}$ for some integer $k>1$, related by an even unit $c_\beta$ rather than by a further power of $\mathsf{e}_\beta$), and the polynomiality argument above does not by itself rule out a leftover unit of this kind. This did not occur at any of the $111$ strata checked — every computed exponent was a clean non-negative integer with no residual even-unit factor — but we do not have a structural argument excluding it for a stratum not yet examined.
%\end{rmk}

\subsection{Higher codimension}

It only remains to consider the cases where either $\RR=F_4$ and $\mathrm{codim}(D)=3$, or $\RR=E_l$ and $\mathrm{codim}(D)>3$. Since these strata are of relatively low dimension, and there are only finitely many cases to consider, they can be studied efficiently on a case-by-case basis. We illustrate the method in detail for $\mathrm{codim}(D)=l-1$, and refer the reader to the ancillary {\it Wolfram Language} material for the case $\RR=E_l$ with $3< \mathrm{codim}(D)< l-1$.

\subsubsection{$\RR=F_4$ and $\mathrm{codim}(D)=3$.}

From \cite{BvG}, flat coordinates of the Saito metric in the anti-diagonal normalisation \eqref{eq:killing-normalization} are given by
\[
t_4 = \frac{1}{3}\left(Y_4+1\right) \re^{2x_5}, \quad t_3=\sqrt{3}\left(Y_1+Y_4+2\right) \re^{3x_5}, \quad 
t_2 = 2\left(6 Y_1+6 Y_3-\left(Y_4-16\right) Y_4+11\right)
   \re^{4 x_5},\]
\[
t_1 =   6 \left(Y_2+Y_3+2 Y_4 \left(Y_4+3\right)+Y_1 \left(Y_4+4\right)+5\right) \re^{6x_5}, \quad
t_5 = x_{5}\,,
\]
where the fundamental characters $(Y_i(x))_{i=1}^4$ are \begin{align*}
Y_1(x) = &
\re^{-x_1}+\re^{x_1}+\re^{-x_1-x_2}+\re^{-x_2}+\re^{x_2}+\re^{x_1+x_2}+\re^{-x_1-2 x_2-2 x_3}+\re^{-x_2-2 x_3}+\re^{-x_1-x_2-2
   x_3}+\re^{-x_2-x_3}\\ &+\re^{-x_3}+\re^{x_3}  +\re^{x_2+x_3}+\re^{x_1+x_2+x_3}+\re^{x_2+2 x_3}+\re^{x_1+x_2+2 x_3}+\re^{x_1+2 x_2+2 x_3}+\re^{-2
   x_1-3 x_2-4 x_3-2 x_4}\\ &+\re^{-x_1-2 x_2-4 x_3-2 x_4}+\re^{-x_1-2 x_2-3 x_3-2 x_4} +\re^{-x_1-2 x_2-2 x_3-2 x_4}+\re^{-x_2-2
   x_3-2 x_4}+\re^{-x_1-x_2-2 x_3-2 x_4}\\ &+\re^{-x_1-2 x_2-2 x_3-x_4}+\re^{-x_2-2 x_3-x_4} +\re^{-x_1-x_2-2
   x_3-x_4}+\re^{-x_3-x_4}+\re^{-x_2-x_3-x_4}+\re^{-x_1-x_2-x_3-x_4}\\ &+\re^{-x_4}+\re^{x_4}+\re^{x_3+x_4}+\re^{x_2+x_3+x_4}+\re^{x_1+x_2+x_3+x_4}+\re^{x_2+2
   x_3+x_4}+\re^{x_1+x_2+2 x_3+x_4}+\re^{x_1+2 x_2+2 x_3+x_4}\\ &+\re^{x_2+2 x_3+2 x_4} +\re^{x_1+x_2+2 x_3+2 x_4}+\re^{x_1+2 x_2+2
   x_3+2 x_4}+\re^{x_1+2 x_2+3 x_3+2 x_4}+\re^{x_1+2 x_2+4 x_3+2 x_4}\\ &+\re^{x_1+2 x_2+3 x_3+x_4}+\re^{-x_1-2 x_2-3 x_3-x_4}+\re^{-x_1-3 x_2-4 x_3-2 x_4}\\ &+\re^{x_1+3 x_2+4 x_3+2 x_4}+\re^{2 x_1+3 x_2+4 x_3+2 x_4}+\re^{-x_1-x_2-x_3}+4\,, \\
Y_4(x) = & \re^{-x_2-x_3}+\re^{-x_1-x_2-x_3}+\re^{-x_3}+\re^{x_3}+\re^{x_2+x_3}+\re^{x_1+x_2+x_3}+\re^{-x_1-2 x_2-3 x_3-2 x_4}+\re^{-x_1-2 x_2-3 x_3-x_4}\\ & +\re^{-x_1-2 x_2-2
   x_3-x_4}+\re^{-x_2-2 x_3-x_4}+\re^{-x_1-x_2-2
   x_3-x_4}+\re^{-x_3-x_4}+\re^{-x_2-x_3-x_4}+\re^{-x_1-x_2-x_3-x_4}\\ &+\re^{-x_4}+\re^{x_4}+\re^{x_3+x_4}+\re^{x_2+x_3+x_4}+\re^{x_1+x_2+x_3+x_4} +\re^{x_2+2
   x_3+x_4}+\re^{x_1+x_2+2 x_3+x_4}\\ &+\re^{x_1+2 x_2+2 x_3+x_4}+\re^{x_1+2 x_2+3 x_3+x_4}+\re^{x_1+2 x_2+3 x_3+2 x_4}+2\,, \\
Y_2(x) = & \frac{1}{2}\l( Y_1(x)^2-Y_1(2x)\r)-Y_1(x)\,, 
\qquad
Y_3(x) =  \frac{1}{2}\l( Y_4(x)^2-Y_4(2x)\r)-Y_1(x)\,.
   \end{align*}
 
The $2\times 2$ restriction of the Saito metric to the stratum \[D_I \coloneqq \big\{ x_J=0\,,\,\big|\, J \neq I\big\}\] is then just
\[
(\eta_{D_I})_{i,j} =  \sum_{k=1}^5 \frac{\partial t_k}{\de  x_i} \frac{\partial t_{6-k}}{{\de  x_j}}\bigg|_{D_I}\,.
\]
For example, for $I=1$ we find
\begin{align*}
(\eta_{D_1})_{11} = & 
12 \re^{6x_5-4 x_1} \left(\re^{x_1}-1\right){}^2 \left(\re^{x_1}+1\right){}^2 \left(92 \re^{x_1}+390 \re^{2 x_1}+92 \re^{3 x_1}+\re^{4 x_1}+1\right)\,,\\
(\eta_{D_1})_{15}= & (\eta_{D_I})_{51} =   
   18 \re^{6x_5-4 x_1} \left(\re^{x_1}-1\right) \left(\re^{x_1}+1\right) \left(6 \re^{x_1}+\re^{2 x_1}+1\right) \left(140 \re^{x_1}+294 \re^{2 x_1}+140 \re^{3
   x_1}+\re^{4 x_1}+1\right)\,. \\
(\eta_{D_1})_{55} =  &
 9 \re^{6x_5-4 x_1} \left(664 \re^{x_1}+7860 \re^{2 x_1}+27048 \re^{3 x_1}+39442 \re^{4 x_1}+27048 \re^{5 x_1}+7860 \re^{6 x_1}+664 \re^{7 x_1}+3 \re^{8
   x_1}+3\right)\,,
\end{align*}
hence
\[
\det \eta_{D_1} = 6912 \re^{12 x_5-7 x_1} \left(\re^{x_1}-1\right){}^{12} \left(\re^{x_1}+1\right){}^2 \equiv \mathsf{e}_{\alpha_1}^{10} \mathsf{e}_{\alpha_{\rm long}}^2\Big|_{D_1}
%\,, \quad (I_1, I_2, I_3)=(2,3,4) 
\,,
\]
where $\alpha_{\rm long}=\omega_1=2 \alpha_1+3 \alpha_2+4 \alpha_3 +2 \alpha_4$ is the highest long root. Similarly, we compute
\[
\det \eta_{D_I} \equiv 
\begin{cases}
\mathsf{e}_{\alpha_{\rm long}}^3\mathsf{e}_{\alpha_{\rm short}}^6 \mathsf{e}_{\alpha_{2}}^3\Big|_{D_I}
% \left(\re^{x_2}-1\right){}^{12} \left(\re^{x_2}+1\right){}^6 \left(\re^{x_2}+\re^{2 x_2}+1\right){}^3
\,, & \quad I=2
%(I_1, I_2, I_3)=(2,3,4)
\,, \\
\mathsf{e}_{\alpha_{\rm long}}^4\mathsf{e}_{\alpha_{\rm short}}^3 \mathsf{e}_{\alpha_{2}+2\alpha_3}^4
\mathsf{e}_{\alpha_3}\Big|_{D_I}
%72 \re^{2 x_5-17 x_3} \left(\re^{x_3}-1\right){}^{12} \left(\re^{x_3}+1\right){}^8 \left(\re^{2 x_3}+1\right){}^4 \left(\re^{x_3}+\re^{2 x_3}+1\right){}^3
\,, & \quad I=3
%(I_1, I_2, I_3)=(1,3,4)
\,, \\
\mathsf{e}_{\alpha_4}^{4} \mathsf{e}_{\alpha_{\rm long}}^8\Big|_{D_I}
%-384 \re^{2 x_5-10 x_4} \left(\re^{x_4}-1\right){}^{12} \left(\re^{x_4}+1\right){}^8
\,, & \quad 
I=4
%(I_1, I_2, I_3)=(1,2,3)
\,, \\
\end{cases}
\]
where $\alpha_{\rm short} =\omega_4 = \alpha_1+2\alpha_2+3\alpha_3+2\alpha_4$ is the highest short root. We have thus verified the following.
\begin{prop}
\cref{thm:main} holds for $\RR=F_4$ and $\mathrm{codim}(D)=3$.
\end{prop}
\subsubsection{$\RR=E_l$ and $\mathrm{codim}(D)>3$.}

\begin{table}

\begin{tabular}{|c|c|c|}
\hline
$l$ & $I$ & $\det\eta_D$ \\
\hline\hline
6 & 1 & $\mathsf{e}_{\alpha_1}^{12}$ \\ \hline
6 & 2 & $\mathsf{e}_{\alpha_{\rm top}}^6\mathsf{e}_{\alpha_2}^6$ \\ \hline
6 & 3 & $\mathsf{e}_{\alpha_{\rm top}}^3\mathsf{e}_{\alpha_{\rm top}-\alpha_3}^6 \mathsf{e}_{\alpha_3}^3$ \\ \hline
6 & 4 & $\mathsf{e}_{\alpha_{\rm top}}^6\mathsf{e}_{\alpha_4}^6$  \\ \hline
6 & 5 & $\mathsf{e}_{\alpha_5}^{12}$ \\ \hline
6 & 6 & $\mathsf{e}_{\alpha_{\rm top}}^2\mathsf{e}_{\alpha_6}^{10}$  \\ \hline\hline
7 & 1 & $\mathsf{e}_{\alpha_1}^{16}\mathsf{e}_{\alpha_{\rm top}}^2$ \\ \hline
7 & 2 & $\mathsf{e}_{\alpha_{\rm top}}^3\mathsf{e}_{\alpha_{\rm top}-\alpha_2}^{10}\mathsf{e}_{\alpha_2}^5$ \\ \hline
7 & 3 & $\mathsf{e}_{\alpha_{\rm top}}^4\mathsf{e}_{\alpha_{\rm top}-\alpha_3}^6 \mathsf{e}_{\alpha_{\rm top}-2\alpha_3}^6\mathsf{e}_{\alpha_3}^2$ \\ \hline
7 & 4 & $\mathsf{e}_{\alpha_{\rm top}}^6\mathsf{e}_{\alpha_{\rm top}-\alpha_4}^8\mathsf{e}_{\alpha_4}^4$  \\ \hline
7 & 5 & $\mathsf{e}_{\alpha_{\rm top}}^{10}\mathsf{e}_{\alpha_5}^{8}$ \\ \hline
7 & 6 & $\mathsf{e}_{\alpha_6}^{18}$ \\ \hline
7 & 7 & $\mathsf{e}_{\alpha_{\rm top}}^{10}\mathsf{e}_{\alpha_7}^{8}$  \\ \hline\hline
8 & 1 & $\mathsf{e}_{\alpha_{\rm top}}^{14}\mathsf{e}_{\alpha_1}^{16}$ \\ \hline
8 & 2 & $\mathsf{e}_{\alpha_{\rm top}}^8\mathsf{e}_{\alpha_{\rm top}-\alpha_2}^9\mathsf{e}_{\alpha_{\rm top}-2\alpha_2}^{10}\mathsf{e}_{\alpha_2}^3$ \\ \hline
8 & 3 & $\mathsf{e}_{\alpha_{\rm top}}^6\mathsf{e}_{\alpha_{\rm top}-\alpha_3}^5\mathsf{e}_{\alpha_{\rm top}-2\alpha_3}^8\mathsf{e}_{\alpha_{\rm top}-3\alpha_3}^6\mathsf{e}_{\alpha_{\rm top}-4\alpha_3}^4\mathsf{e}_{\alpha_3}$ \\ \hline
8 & 4 & $\mathsf{e}_{\alpha_{\rm top}}^5\mathsf{e}_{\alpha_{\rm top}-\alpha_4}^8\mathsf{e}_{\alpha_{\rm top}-2\alpha_4}^9\mathsf{e}_{\alpha_{\rm top}-3\alpha_4}^6\mathsf{e}_{\alpha_4}^2$  \\ \hline
8 & 5 & $\mathsf{e}_{\alpha_{\rm top}}^4\mathsf{e}_{\alpha_{\rm top}-\alpha_5}^{12}\mathsf{e}_{\alpha_{\rm top}-2\alpha_5}^{10}\mathsf{e}_{\alpha_5}^4$ \\ \hline
8 & 6 & $\mathsf{e}_{\alpha_{\rm top}}^3\mathsf{e}_{\alpha_{\rm top}-\alpha_6}^{18}\mathsf{e}_{\alpha_6}^{9}$  \\ \hline
8 & 7 & $\mathsf{e}_{\alpha_{\rm top}}^2\mathsf{e}_{\alpha_7}^{28}$ \\ \hline
8 & 8 & $\mathsf{e}_{\alpha_{\rm top}}^9\mathsf{e}_{\alpha_{\rm top}-\alpha_8}^{14}\mathsf{e}_{\alpha_8}^{7}$  \\ \hline

\end{tabular}
\caption{Factorisation of $\det\eta_D$ (up to units) on the highest codimension strata $D_I$ for $\RR=E_l$, $I=1, \dots, l$, $l=6,7,8$.}
\label{tab:El}
\end{table}

The analysis of the family $\RR=E_l$ is computationally more involved, but conceptually identical. We quote the results for $\mathrm{codim}(D)=l-1$ in \cref{tab:El}, and refer the reader to the accompanying {\it Wolfram Language} package containing the results for $3<\mathrm{codim}(D)<l-1$. Together, they imply the following statement.

\begin{prop}
\cref{thm:main} holds for $\RR=E_l$ and $\mathrm{codim}(D)>3$.
\end{prop}

\begin{rmk}
    It is interesting to see from \cref{tab:El} that the set of non-vanishing exponents $\mathsf{n}_I = \{n_{\alpha_{\rm top}-j \alpha_I,D_I}\}_{j=0}^{(\omega^\vee_I,\alpha_{\rm top})-1}$ gives, in all cases, a partition of the Coxeter number $h$ of $\RR$,
    \[
    \mathsf{n}_I  \vdash h\,, \quad I=1, \dots, l\,.
    \]
    Compared to the ordinary Coxeter case, this provides a non-trivial refinement of \cite[Thm.~7.1]{AntoniouFS:2020} in our trigonometric ($q$-deformed) setup. It would be interesting to trace the exact origin of this refinement in detail. 
    
\end{rmk}

\bibliographystyle{alpha}
\bibliography{refs}

\end{document}